\documentclass[12pt,letterpaper,reqno]{amsart}
\usepackage[foot]{amsaddr}
\makeatletter
\renewcommand{\email}[2][]{%
  \ifx\emails\@empty\relax\else{\g@addto@macro\emails{,\space}}\fi%
  \@ifnotempty{#1}{\g@addto@macro\emails{\textrm{(#1)}\space}}%
  \g@addto@macro\emails{#2}%
}
\makeatother
\usepackage{epsfig}
\usepackage{amsmath}
\usepackage{amssymb}
\usepackage{amsthm}
\usepackage{indentfirst}
\usepackage{xspace}
\usepackage[dvipsnames]{xcolor}
\usepackage{setspace}
\usepackage{verbatim}
\usepackage[letterpaper,margin=1in,headheight=15pt]{geometry}
\usepackage{mathpazo}
\usepackage{booktabs}
\usepackage{float} 
\definecolor{shadecolor}{rgb}{0.85,0.85,0.85}
\usepackage{bibentry}
\usepackage{bm}
\usepackage{hyperref}
\usepackage{tcolorbox}
\definecolor{darkred}{rgb}{0.5,0.15,0.15}
\hypersetup{colorlinks=true,urlcolor=darkred,linkcolor=darkred,citecolor=darkred}
\allowdisplaybreaks

\newtheorem{thm}{Theorem}
\newtheorem{cor}[thm]{Corollary}
\newtheorem{conj}[thm]{Conjecture}
\newtheorem{lem}[thm]{Lemma}
\newtheorem{prop}[thm]{Proposition}

\theoremstyle{remark}
\newtheorem{rem}[thm]{Remark}
\newtheorem{notate}[thm]{Notation}

\theoremstyle{definition}
\newtheorem{defn}[thm]{Definition}

\numberwithin{thm}{section}
\numberwithin{equation}{section}
\numberwithin{figure}{section}

\newcommand{\cB}{\ensuremath{\mathcal B}}
\newcommand{\cL}{\ensuremath{\mathcal L}}

\newcommand{\cM}{\ensuremath{\mathcal M}}

\newcommand{\R}{\ensuremath{\mathbb R}}
\newcommand{\C}{\ensuremath{\mathbb C}}
\newcommand{\CP}{\ensuremath{\mathbb {CP}}}

\newcommand{\Z}{\ensuremath{\mathbb Z}}

\newcommand{\bD}{{\mathbb{D}}}

\newcommand{\cE}{{\mathcal E}}

\newcommand{\delbar}{\ensuremath{\overline{\partial}}}
\newcommand{\zbar}{\ensuremath{\overline{z}}}

\newcommand{\semif}{\ensuremath{\mathrm{sf}}}
\newcommand{\app}{\ensuremath{\mathrm{app}}}

\newcommand{\I}{{\mathrm i}}
\newcommand{\e}{{\mathrm e}}
\newcommand{\de}{\mathrm{d}}

\newcommand{\norm}[1]{\lVert#1\rVert}
\newcommand{\IP}[1]{\left\langle#1\right\rangle}

\newcommand{\eps}{\epsilon}

\newcommand{\del}{{\partial}}

\DeclareMathOperator{\Det}{Det}
\DeclareMathOperator{\End}{End}

\newcommand{\fixme}[1]{{\color{BrickRed}{[#1]}}}

\begin{document}
\onehalfspacing

\title{Exponential decay for the asymptotic geometry of the Hitchin metric}
\author{Laura Fredrickson}
\address{Department of Mathematics, Stanford University}
\email{lfredrickson@stanford.edu}

\begin{abstract}
We consider Hitchin's hyperk\"ahler metric $g_{L^2}$ on
the $SU(n)$-Hitchin moduli space over a compact Riemann surface.
We prove that  the difference between the metric $g_{L^2}$ and a simpler ``semiflat'' hyperk\"ahler metric $g_{\semif}$
is exponentially-decaying along generic rays in the Hitchin moduli space,
as conjectured by Gaiotto-Moore-Neitzke.
\end{abstract}

\maketitle

\setcounter{page}{1}

In this paper, we study the asymptotic geometry
of the hyperk\"ahler metric $g_{L^2}$ on the $SU(n)$-Hitchin moduli space, $\cM$.  Gaiotto-Moore-Neitzke conjecture\cite{GMNhitchin, GMNwallcrossing} an expansion of $g_{L^2}$ in terms of another hyperk\"ahler metric, known as the semiflat metric $g_{\semif}$, that exists because $\cM$ is an algebraic completely integrable system.  

Choose a Higgs bundle $(\delbar_E, \varphi)$ in a generic\footnote{For $SU(2)$, we take $(\delbar_E, \varphi)$ in any non-degenerate fiber. For $SU(n)$, the genericity condition is defined in Definition \ref{defn:simpleeigenvaluecrossing}.} fiber of the $SU(n)$-Hitchin moduli space $\cM$,
and consider the ray $(\delbar_E, t \varphi, h_t)$ in $\cM$ of Higgs bundles $(\delbar_E, t\varphi)$ with harmonic metric $h_t$ on the complex vector bundle $E$, as shown in Figure \ref{fig:ends}.
At $t=1$, fix a Higgs bundle deformation $(\dot{\eta}, \dot{\varphi})$ of $(\delbar_E, \varphi)$; note that $\dot\psi_t =(\dot{\eta}, t \dot{\varphi})$ is a Higgs bundle
deformation of $(\delbar_E, t \varphi)$.
We prove that the difference of the two metrics is exponentially-decaying in $t$.
\begin{figure}[ht]
 \begin{center}
 \includegraphics[height=1in]{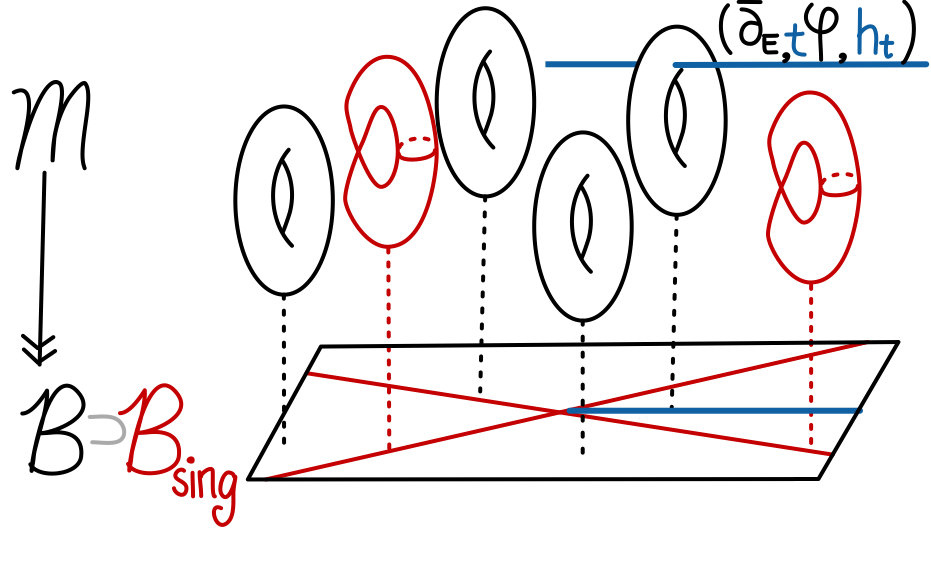}
 \vspace{-.2in}
  \caption{\label{fig:ends}}
 \end{center}
\end{figure}
 \vspace{-.2in}
\begin{thm}\label{thm:main}
Fix a generic Higgs bundle $(\delbar_E, \varphi)$ in $\cM$, and a Higgs bundle variation
$\dot{\psi}=(\dot{\eta}, \dot{\varphi})$. 
Consider the deformation $\dot{\psi}_t=(\dot{\eta}, t \dot{\varphi}) \in T_{(\delbar_E, t \varphi)} \cM$ over the ray $(\delbar_E, t \varphi, h_t)$.
As $t \rightarrow \infty$, the difference between 
 Hitchin's hyperk\"ahler $L^2$-metric $g_{L^2}$ on $\cM$ and the semiflat (hyperk\"ahler) metric $g_{\semif}$ is  exponentially-decaying.  In particular, there is some constant $\gamma>0$, such that 
 \begin{equation} \label{eq:summary}
  g_{L^2}(\dot{\psi}_t, \dot{\psi}_t) = g_{\semif}(\dot{\psi}_t, \dot{\psi}_t) +  O(\e^{-\gamma t}).
 \end{equation}
\end{thm}

\section{Background \& Strategy}\label{sec:background}

In this section,
we review the relevant results about Hitchin's
hyperk\"ahler metric $g_{L^2}$ (\S \ref{sec:L2intro})
and the semiflat metric $g_{\semif}$ (\S \ref{sec:sfintro}).
Additionally, we review Gaiotto-Moore-Neitzke's conjecture \cite{GMNwallcrossing, GMNhitchin}
for Hitchin's hyperk\"ahler metric $g_{L^2}$  
and the progress made towards this conjecture in the rank $2$ case (\S\ref{sec:survey}). Of particular note, Mazzeo-Swoboda-Weiss-Witt \cite{MSWW17} have shown that $g_{L^2}-g_{\semif}$ decays \emph{polynomially} in $t$
along a generic ray in $\cM_{SU(2)}$; Dumas-Neitzke \cite{DumasNeitzke} have shown that---restricted to the image of the Hitchin section of $\cM_{SU(2)}$--- $g_{L^2}-g_{\semif}$ decays \emph{exponentially} in $t$.
We comment on some of the important ingredients in their respective proofs that we gladly borrow, and describe our strategy of proof in \S\ref{sec:compare}. We also highlight some results that we prove in the process about the semiflat metric and Hitchin's hyperk\"ahler metric that we believe are of independent interest and utility.

\subsection{Higgs bundles and the harmonic metric.} \label{sec:Higgsreview}

Let $C$ be a compact Riemann surface with metric $g_C$.
Additionally, let $E \rightarrow C$ be a complex vector bundle of rank $n$
and degree $0$, with some fixed holomorphic structure on 
the determinant line bundle $\Det E$. 
From this data, we get an associated moduli space $\cM$
of polystable  $SL(n,\C)$-Higgs bundles up to complex gauge equivalence.  An $SL(n,\C)$-Higgs bundle is a  pair $(\delbar_E, \varphi)$ where
\begin{itemize}
 \item $\delbar_E$ is a holomorphic structure on $E$ which induces the fixed holomorphic structure $\delbar_{\Det E}$ on $\Det E$, and  
 \item the Higgs field $\varphi \in \Omega^{1,0}(C, \End E)$ is traceless and satisfies $\delbar_E \varphi=0$. 
\end{itemize}
Equivalently, $\cM$ is the $SU(n)$-Hitchin moduli space
consisting of triples $(\delbar_E, \varphi, h)$---up to complex gauge equivalence. Here, the hermitian metric $h$ solves Hitchin's equations
\begin{equation} \label{eq:Hitchin}
 F_{D(\delbar_E, h)} + [\varphi, \varphi^{\dagger_h}]=0,
\end{equation}
and is known as the \emph{harmonic metric}.
In this expression, $D(\delbar_E, h)$ is the Chern connection associated to the pair $(\delbar_E, h)$ , $F_{D(\delbar_E, h)}$ is its curvature, and $\varphi^{\dagger_h} \in \Omega^{0,1}(C, \End E)$ is the $h$-hermitian adjoint of $\varphi$.
Moreover, we fix a compatible\footnote{Here, the compatibility condition is $F_{D(\delbar_{\Det E}, h_{\Det E})}=0$.
} hermitian metric $h_{\Det E}$ on $\Det E$ and insist that $h$ induces this fixed hermitian metric on $h_{\Det E}$.

The Hitchin moduli space $\cM$ is a noncompact hyperk\"ahler space
\begin{eqnarray}
 \mathrm{Hit}: \qquad \quad \cM &\rightarrow& \cB \simeq \bigoplus_{i=2}^n H^0(C, K_C^i) \\ \nonumber
 (\delbar_E, \varphi, h) & \mapsto&  \mathrm{char}_\varphi(\lambda) = \lambda^n - \sum_{i=2}^n \lambda^{n-i} q_i  \qquad q_i \in H^0(C, K_C^i). 
\end{eqnarray}

Given a point of the Hitchin base, the associated 
spectral cover is 
\begin{equation}\Sigma=\{\lambda \in K_C: \mathrm{char}_\varphi(\lambda)=0\} \overset{\pi}{\to} C.
\end{equation}
Let $\cB'$ be the locus of points in $\cB$ where the associated spectral cover is smooth.
Note that when $n=2$, $\mathcal{B} \simeq H^0(C, K_C^2)$; the spectral cover is smooth
if, and only if, $q_2=-\det \varphi$ has only simple zeros.

We restrict our attention to the \emph{regular locus} $\cM'=\mathrm{Hit}^{-1}(\cB')$,
since
both hyperk\"ahler metrics, $g_{L^2}$ and $g_{\semif}$, exist and are smooth there.

\subsection{Hitchin's hyperk\"ahler metric, \texorpdfstring{$\mathbf{g}_{L^2}$}{g\_L2}} \label{sec:L2intro}

Hitchin's hyperk\"ahler metric $g_{L^2}$ on the Hitchin moduli space
is usually expressed in the equivalent unitary formulation of Hitchin's equations.
In this formulation, we additionally fix a hermitian metric $h_0$ on the complex vector bundle $E \rightarrow C$. 
Now, the Hitchin moduli space $\cM$ consists of pairs $(\de_A, \Phi)$, where
\begin{itemize}
 \item $\de_A$ is a $h_0$-unitary connection, and
 \item $\Phi \in \Omega^{1,0}(C, \End E)$, 
\end{itemize}
such that $\delbar_A \Phi=0$ and $F_A + [\Phi, \Phi^{\dagger_{h_0}}]=0$---up to $h_0$-unitary gauge equivalence.

We can go back and forth between these two formulations as follows. 
Clearly, given the pair $(\de_A, \Phi)$ we get the associated triple $(\delbar_A, \Phi, h_0)$. Conversely, given a triple $(\delbar_E, \varphi, h)$,
the group of complex gauge transformations acts by
\begin{equation} \label{eq:complexgaugeequiv}
 g \cdot (\delbar_E, \varphi, h) = ( g^{-1} \circ \delbar_E \circ g, g^{-1} \varphi g, g \cdot h) \qquad \mbox{where $(g \cdot h)(w_1,w_2)=h(g w_1,g w_2)$}.
\end{equation}
There is an $\End E$-valued $h_0$-hermitian section $H$ such that $h(w_1,w_2)=h_0(Hw_1,w_2)$.
Taking $g=H^{-1/2}$, then $g \cdot h=h_0$. The associated pair $(\de_A, \Phi)$ is then $\delbar_A = H^{1/2} \circ \delbar_E \circ H^{-1/2}$ and $\Phi = H^{1/2} \varphi H^{-1/2}$.

\bigskip

Following the exposition in \cite[\S2.3]{DumasNeitzke},
to define Hitchin's hyperk\"ahler metric $g_{L^2}$ , we first equip
the space $\Omega^{0,1}(\mathfrak{sl}(E)) \oplus  \Omega^{1,0}(\mathfrak{sl}(E))$
with the $L^2$-metric given by 
\begin{equation} \label{eq:metric1}
 g \left((\dot{A}^{0,1}_1, \dot{\Phi}_1), (\dot{A}^{0,1}_2, \dot{\Phi}_2)\right)
= 2  \mathrm{Re}  \left(\int_C \IP{\dot{A}^{0,1}_1, \dot{A}^{0,1}_2}_{h_0} + \IP{\Phi_1, \Phi_2}_{h_0} \right), 
\end{equation}
where $\IP{\alpha, \beta} = \mathrm{tr}\left(\alpha \wedge \overline{\star} \beta^{\dagger_{h_0}} \right)$ is a two-form. In particular, if $\beta = \beta_z \de z + \beta_{\zbar} \de \zbar$, then $\overline{\star} \beta^{\dagger_{h_0}}= \beta_z^{\dagger_{h_0}} \overline{\star} \de z + \beta_{\zbar}^{\dagger_{h_0}} \overline{\star} \de \zbar= \I \beta_z^{\dagger_{h_0}} \de \zbar -  \I \beta_{\zbar}^{\dagger_{h_0}} \de z$. (Note that because $C$ is K\"ahler, 
the Hodge star on $1$-forms depends only on the conformal class of the metric $g_C$.)

An infinitesimal $h_0$-unitary gauge transformation $\gamma$ gives rise to a deformation $(\dot{A}_\gamma, \dot{\Phi}_\gamma)$. The map is given by
$\rho: \Omega^0(\mathfrak{su}(E)) \rightarrow \Omega^{0,1}(\mathfrak{sl}(E)) \oplus \Omega^{1,0}(\mathfrak{sl}(E))$ \cite[Eq. 2.5]{DumasNeitzke}, where 
\begin{equation}
 \rho(\gamma) = (\dot{A}_\gamma^{0,1}, \dot{\Phi}_\gamma)=(-\delbar_A \gamma, [\gamma, \Phi])
\end{equation}
We decompose the deformation $(\dot{A}^{0,1}, \dot{\Phi})$
as
\begin{equation}
 (\dot{A}^{0,1}, \dot{\Phi}) = (\dot{A}^{0,1}, \dot{\Phi})^{\parallel} + (\dot{A}^{0,1}, \dot{\Phi})^\perp
\end{equation}
where $(\dot{A}^{0,1}, \dot{\Phi})^{\parallel}$ is parallel to the image of $\rho$, and $(\dot{A}^{0,1}, \dot{\Phi})^\perp$ is perpendicular
to the image of $\rho$.
Then, Hitchin's hyperk\"ahler metric $g_{L^2}$ is
\begin{equation}\label{eq:gL2}
\norm{(\dot{A}^{0,1}, \dot{\Phi})}^2_{g_{L^2}}= \norm{(\dot{A}^{0,1}, \dot{\Phi})^\perp}_g^2,
\end{equation}
where $\norm{\cdot}_g$ is the metric from \eqref{eq:metric1}.
Note that $\norm{(\dot{A}^{0,1}_0, \dot{\Phi}_0)}_{g_{L^2}}$ is the minimum of 
$\norm{(\dot{A}^{0,1}, \dot{\Phi})}_{g}$ among all deformations $(\dot{A}^{0,1}, \dot{\Phi})$ in the unitary gauge orbit of $(\dot{A}^{0,1}_0, \dot{\Phi}_0)$.
Such a minimizing deformation is said to be in Coulomb gauge.

\subsection{Limiting configurations and the semiflat metric \texorpdfstring{$\mathbf{g}_{\semif}$}{g\_sf}} \label{sec:sfintro} 

Consider a $1$-parameter family of Hitchin moduli spaces $\cM_t$, each an upgrade of the usual Higgs bundle moduli space $\cM$. The moduli space $\cM_{t}$ consists of triples $(\delbar_E, \varphi, h_t)$ where $(\delbar_E, \varphi)$ is a Higgs bundle and the hermitian metric  $h_t$ solves the \emph{$t$-rescaled Hitchin's equations}
\begin{equation}
 F_{D(\delbar_E, h_t)} + t^2[\varphi, \varphi^{\dagger_{h_t}}]=0.
\end{equation}
Note that $(\delbar_E, \varphi, h_t)$ solves the $t$-rescaled Hitchin's equations if, and only if, $(\delbar_E, t \varphi, h_t)$ solves the usual $t=1$ Hitchin's equations in \eqref{eq:Hitchin}. Consequently, the hermitian metric $h_\infty$ appearing in $\cM_\infty$ is the limiting metric $h_\infty=\lim_{t \to \infty} h_t$ of the family of hermitian metrics in the ray $(\delbar_E, t \varphi, h_t)$ in $\cM$.  We call the triple $(\delbar_E, \varphi, h_\infty)$ in  $\cM_\infty$ a \emph{limiting configuration}.

For $(\delbar_E, \varphi) \in \cM'$, the limiting metric $h_\infty$ is obtained by pushing forward the singular Hermitian-Einstein metric on a certain parabolic line bundle $\cL \rightarrow \Sigma$ \cite{MSWW14, FredricksonSLn}. 
Here, $\cL \rightarrow \Sigma$ is the spectral data associated to the Higgs bundle $(\delbar_E,\varphi)$; it encodes the eigenvalues (as $\Sigma$) and eigenspaces (as $\cL$) of $\varphi: \cE \rightarrow \cE \otimes K_C$, where $\cE=(E, \delbar_E)$.
The spectral curve $\pi: \Sigma \to C$ is ramified at $Z \subset C$.
\begin{figure}[ht]
 \begin{center}
 \includegraphics[height=.75in]{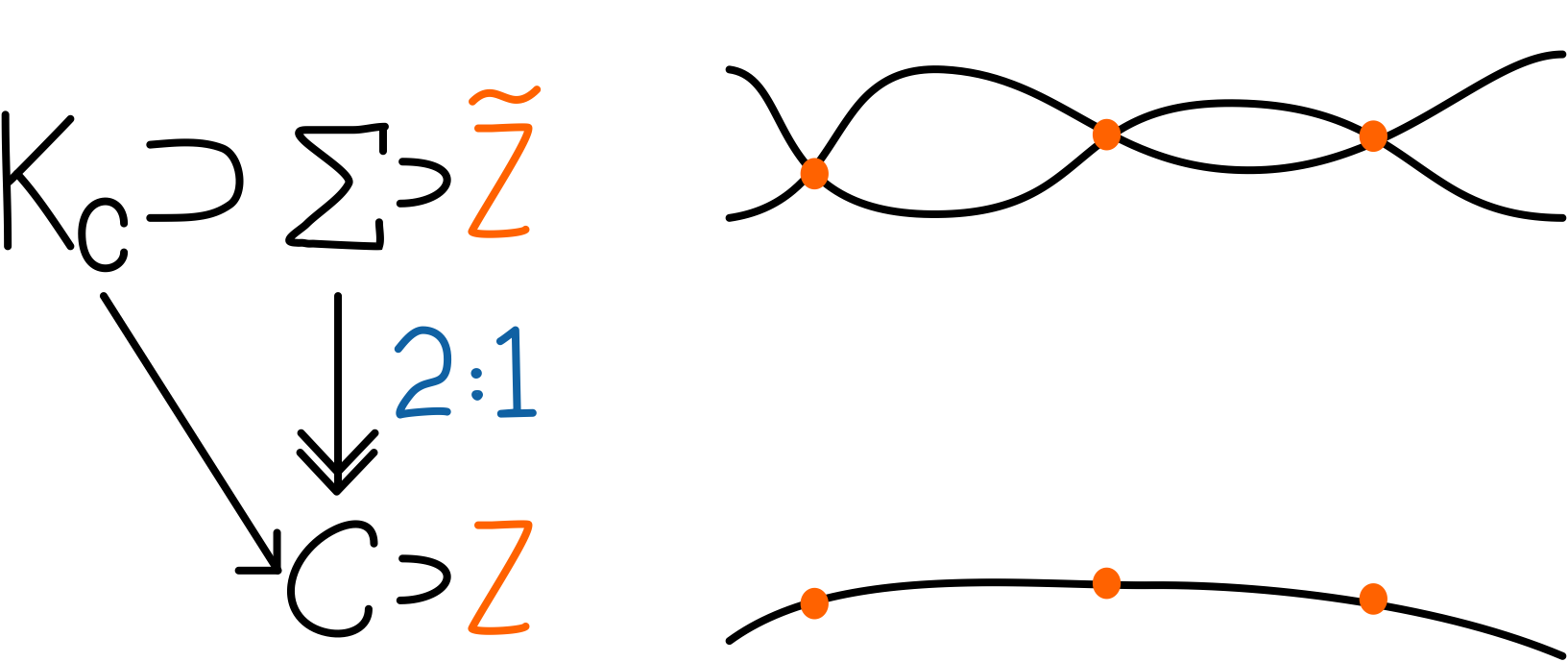}
  \caption{\label{fig:ram}For $SU(2)$, the spectral cover $\Sigma$ is an $2:1$ cover of $C$, ramified at $Z$, the zeros of $\det \varphi$.}
 \end{center}
\end{figure}
For $SU(2)$, we equip the line bundle  $\cL \rightarrow \Sigma$ with weights $-\frac{1}{2}$ at ramification divisor $\widetilde{Z} \subset \Sigma$, which is the set preimages of the branch divisor $Z\subset C$, as shown in Figure \ref{fig:ram}. Then, $h_\infty$ is the orthogonal pushforward of the Hermitian-Einstein metric on this parabolic bundle.  For $SU(n)$ with $n\geq 3$, the construction is similar, and the parabolic weights are described in \cite{FredricksonSLn}.

 Because $h_\infty$ is pushed forward from abelian data, $(\delbar_E, \varphi, h_\infty)$ solves the \emph{decoupled Hitchin's equations}
\begin{equation}
 F_{D(\delbar_E, h_\infty)} =0, \quad [\varphi, \varphi^{\dagger_{h_\infty}}]=0.
\end{equation}
The moduli space of limiting configurations $\cM'_\infty$ has a natural $L^2$-metric $g_{L^2}(\cM'_\infty)$. (Here, the Coulomb gauge condition is only formal.)  

As an algebraic completely integrable system, the Hitchin moduli space  has a natural hyperk\"ahler metric $g_{\semif}$, known as the semiflat metric, that is smooth on the regular locus, $\cM'$\cite[Theorem 3.8]{Freedsemiflat}. In \S\ref{sec:semi}, we review the general definition of the semiflat metric, and give a complete characterization of the semiflat metric on the Hitchin moduli space in Proposition \ref{prop:sfcharacterization}.
In the $SU(2)$ case, Mazzeo-Swoboda-Weiss-Witt prove that these two metrics are equal.
\begin{prop}\label{prop:MSWW} \cite{MSWW17} 
The semiflat metric $g_{\semif}$ is the natural $L^2$-metric on the moduli space of limiting configurations $\cM'_\infty$, for deformations in formal Coulomb gauge.
\end{prop}
We extend this result to $SU(n)$ in Theorem \ref{thm:semiflatisL2}.

\subsection{Survey of Previous Results}\label{sec:survey}

\subsubsection{Gaiotto-Moore-Neitzke's conjecture}\label{sec:GMNsurvey}
In \cite{GMNhitchin, GMNwallcrossing}, Gaiotto-Moore-Neitzke conjecture that 
Hitchin's hyperk\"ahler metric solves an integral relation\footnote{The hyperk\"ahler metric $g_{L^2}$ on $\cM$ determines and is completely
determined by a twisted fiber-wise holomorphic symplectic structure on the twistor space $\mathcal{Z}=\cM \times \CP^1$.  The integral relation is actually formulated in terms of certain ``holomorphic Darboux coordinates'' $\mathcal{X}_\gamma$ on the twistor space $\mathcal{Z}=\cM \times \CP^1$}. (The integral equation appears in \cite[Eq 4.8]{Neitzkehyperkahler} 
which is a survey of \cite{GMNhitchin} aimed at mathematical audiences.)
Iterating the integral relation from the initial hyperk\"ahler metric $g_{\semif}$, one expects to approach
the hyperk\"ahler metric $g_{L^2}$.  The first iteration gives the following expression for the difference of the two metrics over a ray $t b \in \mathcal{B}$:

\begin{equation} \label{eq:firstit}
 g_{L^2} = g_{\semif} - \frac{2t^2}{\pi} \sum_{\gamma \in \Gamma_b} \Omega(\gamma; b) K_0(2t|Z_\gamma|) \de |Z_\gamma|^2 + \ldots
\end{equation}
where $\Sigma_b \overset{\pi}{\rightarrow} C$ is the spectral cover, $\Gamma_b$ is a sublattice\footnote{\label{footnote:Bessel}In the $SU(2)$ case, $\Gamma_b=H^{\mathrm{odd}}_1(\Sigma_b, \Z)$, the sublattice of $H_1(\Sigma_b,\Z)$ which is odd under the exchange of the two sheets of $\Sigma$. For $SU(n)$, $\Gamma_b = H_1(\Sigma_b, \Z)_{\sigma}$, a sublattice defined in \eqref{eq:sigma}} of $H_1(\Sigma_b, \Z)$, $K_0$ is the modified Bessel function of the second kind\footnote{The function $K_0(x)$ solves the modified Bessel differential equation $x^2 y''(x) + x y'(x) - x^2 y(x) =0$ on $(0, \infty)$.  Within the two-dimensional family of solutions, the function $K_0(x)$ is determined (up to multiplication by constant) by the property that $\lim_{x \to \infty} K_0(x)=0$.  It is defined by the integral $K_0(x) = \int_0^\infty \frac{\cos(xt)}{\sqrt{t^2+1}} \de t$.}, $\Omega(\gamma; b)$ is an integer-valued generalized Donaldson-Thomas invariant\footnote{
Let $\Gamma \rightarrow \cB'$ be the lattice bundle with fiber $\Gamma_b$.
Then, $\Omega: \Gamma \rightarrow \Z$.  Given a section $\gamma$ of $\Gamma \rightarrow \cB'$, the function $\Omega(\gamma; \; \cdot \;): \cB' \rightarrow \Z$ is typically not continuous.  The value jumps at real-codimension-1 walls in $\cB'$ and satisfies the Kontsevich-Soibelman wall-crossing formula.
} (see \cite{kontsevichsoibelman, JoyceSong} and the discussion in \cite{Neitzkehyperkahler}) depending on $b \in \mathcal{B}$, and $Z$ is the period map 
\begin{equation}
 Z: \Gamma_b \rightarrow \C \qquad Z_\gamma = \oint_\gamma \lambda,
\end{equation}
where $\lambda$ is the tautological (Liouville) $1$-form on $\mathrm{Tot}(K_C)$.
The first correction is from the smallest value $2|Z_{\gamma_0}|$ for which $\Omega(\gamma_0;b) \neq 0$.
Since $K_0(x) \sim \sqrt{ \frac{\pi}{2x}} \e^{-x}$, the first correction $K_0(2t|Z_{\gamma_0}|)$ is of order $O \left(\e^{-2|Z_{\gamma_0}|t} \right)$. The omitted cross-terms in \eqref{eq:firstit} are of order $O \left(\e^{-4|Z_{\gamma_0}|t}\right)$ \cite[Eq. 5.3]{Neitzkenotes}.
\begin{conj}[Weak form of Gaiotto-Moore-Neitzke's conjecture for $\cM_{SU(n)}$] \label{conj:weakGMN}
Fix a Higgs bundle $(\delbar_E, \varphi)$ in $\cM'$.
 Hitchin's $L^2$-metric on $\cM'$ admits an expansion as
 \begin{equation} \label{eq:weakGMN}
 g_{L^2} = g_{\semif} + O \left(\e^{-2|Z_{\gamma_0}|t} \right).
\end{equation}
In the $SU(2)$ case, $|Z_{\gamma_0}|=2M$ where $M$ is the length of a shortest geodesic on the associated spectral cover $\Sigma$, measured in the singular flat metric $\pi^*|\det \varphi|$.
\end{conj}

A number of previous papers made progress towards proving 
this conjecture for $SU(2)$.

\subsubsection{Mazzeo-Swoboda-Weiss-Witt's description of the ends of \texorpdfstring{$\cM'_{SU(2)}$}{M'} using approximate solutions}\label{sec:MSWWsurvey}
In \cite{MSWW17}, Mazzeo-Swoboda-Weiss-Witt prove that $g_{L^2}$ admits an  expansion as
\begin{equation}\label{eq:MSWWexpansion}
   g_{L^2}(\dot{\psi}_t, \dot{\psi}_t) = g_{\semif}(\dot{\psi}_t, \dot{\psi}_t) +  \sum_{j=0}^\infty t^{(4-j)/3} G_j +O(\e^{-\gamma t}),
\end{equation}
where $G_j$ is a symmetric $2$-tensor and $\psi_t=(\dot{\eta}, t \dot{\varphi})$ is an infinitesimal deformation of $(\delbar_E, t \varphi)$.  (Since $\norm{\psi_t}^2$ grows quadratically in $t$ the first possible polynomial term is actually of order $t^{-2/3} \norm{\dot \psi_t}^2$. Thus, it is polynomially-decaying even though
it appears in \eqref{eq:MSWWexpansion} to be polynomially-growing.)

The result in \cite{MSWW17} is built on the description of $h_t$ near the ends
of $\cM'$ given in \cite{MSWW14}. 
They build an approximate solution of Hitchin's equations 
$h_t^\app$ close to $h_t$ \cite[Theorem 6.7]{MSWW14}
by desingularizing the singular metric $h_\infty$ on disks $\mathbb{D}$ around the zeros of $q_2$.  In particular, around each zero of $q_2$, there is a local coordinate $z$ in which $q_2=z \de z^2$ and a holomorphic frame in which
 \begin{equation} \label{eq:shape}
  \delbar_E =\delbar \qquad \varphi = \begin{pmatrix}0 & z \\ 1 & 0 \end{pmatrix} \de z
  \qquad \left.h_\infty \right|_{\mathbb{D}} = \begin{pmatrix} |z|^{-1/2} & 0\\ 0& |z|^{1/2} \end{pmatrix}.
 \end{equation}
The metric has a desingularization given by 
\begin{equation}\label{eq:htmodel}
\left. h^{\mathrm{model}}_t \right|_{\mathbb{D}} = \begin{pmatrix} |z|^{-1/2}\e^{-u_t(|z|)} & 0\\ 0& |z|^{1/2}\e^{u_t(|z|)} \end{pmatrix},
\end{equation}
where $u_t: \R^{+} \rightarrow \R$
solves
 \begin{equation} \label{eq:umodel}
\left(\frac{\de^2}{\de |z|^2} + \frac{1}{|z|} \frac{\de}{\de |z|} \right) u_t - 8t^2|z|\sinh(2u_t)=0,
\end{equation}
with asymptotics \\
\begin{tabular}{ l l l }\label{eq:asymptotics}
 \hspace{2in} &$u_t(|z|) \sim  \frac{1}{\pi} K_0(\frac{8t}{3}|z|^{\frac{3}{2}})$ & as $|z|\rightarrow \infty$ \\
 &$u_t(|z|) \sim - \frac{1}{2} \log (|z|)$
 & as $|z| \rightarrow 0$.
\end{tabular}\\
The approximate metric is built by interpolating between $h^{\mathrm{model}}_t$ on the disks and $h_\infty$ on the complement,
simply by adding the cutoff function $\chi$
\begin{equation} \label{eq:htapprox}
 \left.h^{\mathrm{app}}_t \right|_{\mathbb{D}} = \begin{pmatrix} |z|^{-1/2}\e^{-u_t \chi} & 0\\0 & |z|^{1/2}\e^{u_t \chi} \end{pmatrix}.
\end{equation}
The metric $h^{\app}_t$ fails to solve Hitchin's equations on the annulus where $\chi$ is not equal to $0$ or $1$.
Thus, they call the moduli space of triples $(\delbar_E, t \varphi, h_t^{\mathrm{app}})$ 
the ``approximate Hitchin moduli space.'' (Note that in their construction of approximate
solutions, they use the fact that Hitchin's equations are conformal, and replace $g_C$ with   a conformal metric $g_C'$ which is flat on the disks around each of the zeros of $q_2$.)

Turning back to the difference between Hitchin's hyperk\"ahler metric and the semiflat metric,
Mazzeo-Swoboda-Weiss-Witt break the difference into the following pieces and separately analyze each:
 \begin{eqnarray} 
  g_{L^2} - g_{\semif}
&=&
   \left(g_{L^2}
   - g_{\app}
   \right)
   + \left(  g_{\app}
   - g_{\semif}
   \right).
 \end{eqnarray}
Here, $g_{\app}$ is the $L^2$-metric on the ``approximate Hitchin moduli space''
for deformations in formal Coulomb gauge.
Note that all of their possible polynomial terms come
from the second term, $g_{\app}
- g_{\semif}
$.  Moreover, since $h_t^{\app}=h_\infty$ on the complement of the disks, the integrand of the difference of the $L^2$-metrics 
$g_{\app}
- g_{\semif}
$ is non-zero only on the disks.

\subsubsection{Dumas-Neitzke's description of the asymptotic geometry of the Hitchin section of \texorpdfstring{$\cM'_{SU(2))}$}{M'}}\label{sec:DNsurvey}
In \cite{DumasNeitzke}, Dumas-Neitzke restrict from the regular locus of the Hitchin moduli space $\cM'$ to the image of the Hitchin section with its tangent space. They prove that $g_{L^2}-g_{\semif}$ is exponentially-decaying, and in particular
 \begin{equation} \label{eq:DNsummary}
  g_{L^2}(\dot{\psi}_t, \dot{\psi}_t) = g_{\semif}(\dot{\psi}_t, \dot{\psi}_t) +  O(\e^{-2M t}),
 \end{equation}
 where $M$ is the length of a shortest geodesic
 on $\Sigma$ measured in the singular flat metric $\pi^*|\det \varphi|$.
 Comparing this to the weak form of Gaiotto-Moore-Neitzke's conjecture given in Conjecture \ref{conj:weakGMN}, note that this is \emph{only} off from the conjectured sharpest coefficient of exponential decay by a factor of $2$!
 
 Though Dumas-Neitzke do not use the approximate solutions in \cite{MSWW14},
Dumas-Neitzke essentially have a very clever way of dealing with the term
$g_{\app}(\dot{\psi}_t, \dot{\psi}_t) - g_{\semif}(\dot{\psi}_t, \dot{\psi}_t)$ on the disks.  
The possible polynomial terms in \eqref{eq:MSWWexpansion} are from variations in which the zeros of the quadratic differential $\det(\varphi + \eps \dot{\varphi})$ move.
Dumas-Neitzke use a local biholomorphic flow on the disks around each zero of $q_2$ that perfectly matches the changing location of the zero of $q_2+ \eps \dot{q_2}$.  Moreover, surprisingly,
the most seemingly problematic piece of the integrand for the difference $g_{\app}(\dot{\psi}_t, \dot{\psi}_t) - g_{\semif}(\dot{\psi}_t, \dot{\psi}_t)$ turns out to be an exact form that they can control.  See  \S\ref{sec:near} for a more detailed exposition of Dumas-Neitzke's work.

\subsection{Strategy in Present Work}\label{sec:compare}

For notational convenience and to better
elucidate the relation of the present
work with Dumas-Neitzke and Mazzeo-Swoboda-Weiss-Witt,  we first prove the result for $SU(2)$ in \S\ref{sec:rank2} and then
prove the extension to $SU(n)$ in \S\ref{sec:extension}.

Compared to these two previous papers,
our proof for $n=2$ can be seen as an extension
of the method of Dumas-Neitzke \cite{DumasNeitzke}
using the analysis and approximate solutions of Mazzeo-Swoboda-Weiss-Witt. (In the extension to $SU(n)$, we make use of the construction of approximate solutions $h_t^{\app}$ in \cite{FredricksonSLn}.)  As in \cite{MSWW17}, we break the integral into two parts
 \begin{eqnarray} 
  g_{L^2}(\dot{\psi}_t, \dot{\psi}_t) - g_{\semif}(\dot{\psi}_t, \dot{\psi}_t) &=&
   \left(g_{L^2}(\dot{\psi}_t, \dot{\psi}_t) - g_{\app}(\dot{\psi}_t, \dot{\psi}_t) \right)\\  \nonumber & &  + \left(  g_{\app}(\dot{\psi}_t, \dot{\psi}_t) - g_{\semif}(\dot{\psi}_t, \dot{\psi}_t) \right).
 \end{eqnarray}
Using the analysis in \cite{MSWW14}, Mazzeo-Swoboda-Weiss-Witt prove that the first term  $g_{L^2}-g_{\app}$ decays like $\e^{-\gamma t}$. (Since their coefficient of exponential decay is not some known fraction of the conjectural sharp $-4Mt$, we are not careful about the exact coefficient of exponential decay.) 
We prove that the second term $g_{\app}-g_{\semif}$ is exponentially-decaying in \S\ref{sec:near} using an adaptation of the method of Dumas-Neitzke.

\medskip

We give a formulation of the hyperk\"ahler metric in terms of $SL(n,\C)$-Higgs bundle deformations in \S\ref{sec:Higgsdef}.
Typically, the $L^2$-metric on the Hitchin moduli space is
expressed in the unitary formulation of Hitchin's equations in terms of pairs $(\dot{A}^{0,1}, \dot{\Phi})$.
However, in the unitary formulation, the diffeomorphisms between
 $\cM'$, $\cM'_{\app}$ and $\cM'_{\infty}$ are obscured.
It is far easier to formulate everything in terms of triples $(\delbar_E, \varphi, h)$,
since the same Higgs bundle moduli space underlies $\cM'$, $\cM'_{\app}$ and $\cM'_{\infty}$.
In \S\ref{sec:Higgsdef}, we formulate the tangent space in terms of triples $(\dot{\eta}, \dot{\varphi}, \dot{\nu})$ where $\dot{\eta}$ is a variation of holomorphic structure,
$\dot{\varphi}$ is a variation of the Higgs field, and $\dot{\nu}$ is a variation of the 
hermitian metric determined from $(\dot{\eta}, \dot{\varphi})$ by the equation  \eqref{eq:ingaugetriple}.  Notably, in Proposition \ref{prop:metricdiff}
we give the following clean expression for Hitchin's hyperk\"ahler metric on $\cM$.
Given $(\dot{\eta}, \dot{\varphi}, \dot{\nu}) \in T_{(\delbar_E, \varphi, h)} \cM$, 
\begin{equation}\label{eq:highlight1}
 \norm{(\dot{\eta}, \dot{\varphi}, \dot{\nu})}_{g_{L^2}}^2
   = 2 \int_C \IP{ \dot{\eta} - \delbar_E \dot{\nu},  \dot{\eta}}_h  +
 \IP{\dot{\varphi} +[\dot{\nu}, \varphi], \dot{\varphi}}_h.
\end{equation}
Here, as explained more fully in \eqref{eq:metric1}, $\IP{ \dot{\eta} - \delbar_E \dot{\nu},  \dot{\eta}}_h=\mathrm{tr} \left( (\dot{\eta} - \delbar_E \dot{\nu}) \wedge \overline{\star} \dot{\varphi}^{\dagger_h} \right)$ is a two-form, where $\star$ is the Hodge star  for the metric $g_C$ on $C$.
Note that the fixed background  metric $h_0$ on $E$ characteristic of the unitary formulation of Hitchin's hyperk\"ahler metric does not appear.  
This clean expression uses a somewhat surprising fact (Proposition \ref{prop:coulombgauge}):
we can write a single $\mathfrak{sl}(E)$-valued equation for $(\dot{A}^{0,1}, \dot{\Phi})$ combining (1) the infinitesimal variation of $F_A+[\Phi, \Phi^{\dagger_{h_0}}]=0$  and (2) the Coulomb gauge condition. This is the generalization of and explanation for
 the useful complex scalar PDE featured in \cite[Eq. 4.14]{DumasNeitzke}.
 
 The lengthy \S\ref{sec:Higgsdef} contains two more notable results: In Proposition 
  \ref{prop:sfcharacterization} we give a practical characterization of the semiflat metric. 
 In Theorem \ref{thm:semiflatisL2}, we prove that the $L^2$-metric on the moduli space of limiting
configurations $\cM'_\infty$ agrees with the semiflat metric for $SU(n)$, generalizing Mazzeo-Swoboda-Weiss-Witt's proof for $SU(2)$ in \cite{MSWW17}.  Our proof is simpler, making use of the spectral cover and the analog of \eqref{eq:highlight1} for $\cM'_\infty$.

 \begin{rem}
 We try to use the notation of Mazzeo-Swoboda-Weiss-Witt and Dumas-Neitzke for ease of comparison. However, there is one point of divergence that has the possibility to be especially confusing, since we instead use the notation in \cite{FredricksonSLn, FredricksonNeitzke}. The local harmonic metric in the local holomorphic frame of \eqref{eq:shape} is written respectively
 in these various papers as
 \begin{equation*}
\left. h^{\mathrm{model}}_t \right|_{\mathbb{D}} = \underbrace{\begin{pmatrix} |z|^{-\frac{1}{2}}\e^{-h^{MSWW}_t} &0 \\0 & |z|^{\frac{1}{2}}\e^{h_t^{MSWW}} \end{pmatrix}}_{\mbox{\cite[Eq. 21]{MSWW14}}}=\underbrace{
\begin{pmatrix} \e^{-u_t^{DN}} & 0\\0 & \e^{u_t^{DN}} \end{pmatrix}}_{\mbox{\cite[Eq. 4.1]{DumasNeitzke}}}=\underbrace{
\begin{pmatrix} |z|^{-\frac{1}{2}}\e^{-u_t} &0 \\0 & |z|^{\frac{1}{2}}\e^{u_t} \end{pmatrix}}_{\mbox{\cite{FredricksonSLn, FredricksonNeitzke},\eqref{eq:htmodel}}},
\end{equation*}
so that our $u_t = u_t^{DN} - \frac{1}{2} \log |z|=h_t^{MSWW}$.
\end{rem}

\subsection{Acknowledgements}
The author thanks Rafe Mazzeo,  Andy Neitzke, Hartmut Weiss, and \'Akos Nagy for helpful discussions
related to this work, and Georgios Kydonakis for his comments on an earlier version.
The author also thanks the anonymous referee for helpful comments.

\section{Hyperk\"ahler metrics in terms of Higgs bundle deformations}\label{sec:Higgsdef}

There is a canonical diffeomorphism $\cM' \rightarrow \cM'_\infty$. It is most explicitly given in terms of triples
\begin{eqnarray}
\mathcal{F}: \qquad  \quad \cM' &\rightarrow&  \cM_\infty'\\ \nonumber
[(\delbar_E, \varphi, h)] &\mapsto& [(\delbar_E, \varphi, h_\infty)],
\end{eqnarray}
because the Higgs bundle moduli space underlies both $\cM'=\{(\delbar_E, \varphi, h)\}/_\sim$ and $\cM'_\infty=\{(\delbar_E, \varphi, h_\infty)\}/_\sim$.
The hyperk\"ahler metric $g_{L^2}$ is naturally a metric on $\cM'$, while $g_{\semif}$ is most naturally a metric on $\cM'_\infty$. (In \S\ref{sec:semi}, we prove the semiflat metric is the natural $L^2$-metric on $\cM'_\infty$, generalizing Mazzeo-Swoboda-Weiss-Witt's result for $SU(2)$.)
Consequently, to compare the two hyperk\"ahler metrics $g_{L^2}$ and $\mathcal{F}^*g_{\semif}$ on $\cM'$, we 
must describe the map from the tangent space to $\cM'$ 
to the tangent space to $\cM'_\infty$.

Even though the metrics $g_{L^2}$ and $g_{\semif}$ are normally written in terms of unitary deformations $(\dot{A}^{0,1}, \dot \Phi)$ and $(\dot{A}^{0,1}_\infty, \dot \Phi_\infty)$, it will be convenient to express them  instead in terms of 
the underlying Higgs bundle deformation $(\dot{\eta}, \dot{\varphi})$,
and respective $\mathfrak{sl}(E)$-valued metric variations $\dot{\nu}$ and $\dot{\nu}_\infty$, which we introduce respectively in \S\ref{sec:tangentspace} and \S\ref{sec:othertangentspaces}. We express $g_{L^2}$ in terms of the triple $(\dot{\eta}, \dot{\varphi}, \dot{\nu})$ in Proposition \ref{prop:metricdiff} and  we express the semiflat metric in terms of the triple $(\dot{\eta}, \dot{\varphi}, \dot{\nu}_\infty)$ in Theorem \ref{thm:semiflatisL2}. The induced canonical map is
\begin{eqnarray}
\de \mathcal{F}: \qquad T_{(\delbar_E, \varphi, h)} \cM' &\rightarrow&  T_{(\delbar_E, \varphi, h_\infty)}\cM_\infty' \\ \nonumber
 [(\dot{\eta}, \dot{\varphi}, \dot{\nu})] &\mapsto&  [(\dot{\eta}, \dot{\varphi}, \dot{\nu}_\infty)]. 
\end{eqnarray}
In the main theorem, Theorem \ref{thm:main}, we prove that
\begin{equation}
 \norm{(\dot{\eta}, t\dot{\varphi}, \dot{\nu}_t)}_{g_{L^2}}^2 -  \norm{\underbrace{(\dot{\eta}, t\dot{\varphi}, \dot{\nu}_\infty)}_{\de \mathcal{F} ((\dot{\eta}, t\dot{\varphi}, \dot{\nu}_t))}}_{g_{\semif}}^2 = O(\e^{-\gamma t}).
\end{equation}

\begin{rem}The results in this entire section hold for $G=SU(n)$.
\end{rem}

\subsection{The tangent space to \texorpdfstring{$\cM$}{M}}  \label{sec:tangentspace}

Consider the family of deformations of the Higgs bundle $(\delbar_E, \varphi)$ given by
\begin{eqnarray} \label{eq:defhol}
 (\delbar_E)_\eps &=& \delbar_E + \eps \dot \eta\\ \nonumber
 \varphi_\eps &=& \varphi + \eps \dot \varphi
\end{eqnarray}
where $\dot \eta$ is $(0,1)$-valued in $\End E$ (giving an infinitesimal deformation of the holomorphic structure), $\dot{\varphi}$ is $(1,0)$-valued in $\End E$, and 
\begin{equation}
 \delbar_E \dot \varphi + [\dot \eta, \varphi]=0,
\end{equation}
so that $(\dot{\eta}, \dot{\varphi})$ solves the infinitesimal version of
the Higgs bundle equation $\delbar_E \varphi=0$.
Two infinitesimal Higgs bundle deformations determine
the same class if there is an infinitesimal gauge transformation $\dot{\gamma} \in H^2(\End E)$ such that
\begin{eqnarray} \label{eq:infinitesimal}
 \dot{\eta}_2 - \dot{\eta}_1 &=& \delbar_E \dot{\gamma}\\ \nonumber
 \dot{\varphi}_2 - \dot{\varphi}_1 &=& [\varphi, \dot{\gamma}].
\end{eqnarray}

Now, consider a triple $(\delbar_E, \varphi, h)$ with Higgs bundle deformation $(\dot{\eta}, \dot{\varphi})$. Additionally deform the hermitian metric by including an $\mathfrak{sl}(E)$-valued section $\dot{\nu}$:
\begin{equation}
 h_\eps(w_1,w_2) = h(\e^{\eps \dot{\nu}} w_1, \e^{\eps \dot{\nu}} w_2).
\end{equation}
Note that $h_\eps(w_1,w_2) = h(w_1, w_2) + \eps \left( h( \dot{\nu} w_1, w_2) + h(w_1, \dot{\nu}w_2) \right) + O(\eps^2)$. We see that $[(\dot{\eta}_1, \dot{\varphi}_1, \dot{\nu}_1)]=[(\dot{\eta}_2, \dot{\varphi}_2, \dot{\nu}_2)]$ if there is an infinitesimal gauge transformation $\dot{\gamma} \in L^{2,2}(\End E)$ such that \eqref{eq:infinitesimal} holds and additionally
\begin{equation}\label{eq:infinitesimalnu}
\dot{\nu}_2 -\dot{\nu}_1 = \dot{\gamma}.
\end{equation}
\medskip
We now express this deformation in the more standard 
unitary formulation of Hitchin's equations, as described in \S\ref{sec:L2intro}. 
Given the triple $(\delbar_E, \varphi, h)$, there is an $\End E$-valued $h_0$-hermitian section $H$ such that $h(w_1,w_1)=h_0(Hw_1,w_2)$. 
Consequently, 
\begin{equation}h_\eps(w_1,w_2)=h(\e^{\eps \dot{\nu}} w_1, \e^{\eps \dot{\nu}} w_2) = h_0(\e^{\eps \dot{\nu}^{\dagger_{h_0}}} H \e^{\eps \dot{\nu}} \;  w_1, w_2)
\end{equation}
Taking the gauge transformation $g_\eps = \e^{-\eps \dot{\nu}} H^{-1/2} $ and gauge action in
\eqref{eq:complexgaugeequiv}, note that
\begin{equation}
 g_\eps \cdot ((\delbar_E)_\eps, \varphi_\eps , h_\eps) = (g_\eps^{-1} \circ (\delbar_E)_\eps \circ g_\eps, g^{-1}_\eps \varphi_\eps g_\eps, h_0),
\end{equation}
consequently, we've passed to the unitary formulation of Hitchin's equations.
Then 
\begin{equation}
 g_\eps^{-1} \circ (\delbar_E)_\eps \circ g_\eps=\delbar_A + \eps \dot{A}^{0,1} + O(\eps^2) \qquad  g^{-1}_\eps \varphi_\eps g_\eps =  \Phi + \eps \dot{\Phi} +O(\eps^2)
\end{equation}
where
\begin{equation} \label{eq:defun}
 \dot A^{0,1}
=H^{1/2}  \left( \dot{\eta} - \delbar_E \dot{\nu} \right) H^{-1/2} \qquad \dot{\Phi} = H^{1/2} \left(\dot{\varphi} + [\dot{\nu}, \varphi] \right)  H^{-1/2}.
\end{equation}
is the unitary deformation associated to $(\dot{\eta}, \dot{\varphi}, \dot{\nu})$.

\subsection{Representatives of the tangent space to \texorpdfstring{$\cM'$}{M'} in Coulomb gauge} \label{sec:ingauge}

In the unitary formulation of Hitchin's equations, given $(\dot{A}^{0,1}, \dot{\Phi}) \in T_{(A^{0,1}, \Phi)} \cM$, the pair $(\dot{A}^{0,1}, \dot{\Phi})$ must solve the following three equations, which respectively
encode the infinitesimal version of $\delbar_A \Phi=0$, $F_A + [\Phi, \Phi^{\dagger_{h_0}}]=0$ 
and the formal Coulomb gauge condition (see \cite[Eq. 10]{MSWW17}, \cite[Eq. 4.7]{DumasNeitzke}):
\begin{eqnarray} \label{eq:three}
 \delbar_A \dot{\Phi} + [\dot{A}^{0,1}, \Phi]&=&0\\ \nonumber
 \de_A \dot{A} +[\Phi, \dot{\Phi}^{\dagger_{h_0}}] + [\dot{\Phi}, \Phi^{\dagger_{h_0}}]&=&0\\ \nonumber
 \de_A \star \dot{A} - [\dot{\Phi} - \dot{\Phi}^{\dagger_{h_0}}, \star (\Phi - \Phi^{\dagger_{h_0}})] &=&0.
\end{eqnarray}
Note that the Coulomb gauge condition and infinitesimal version of $F_A +[\Phi, \Phi^{\dagger_{h_0}}]=0$ each impose
an $\mathfrak{su}(E)$-valued equation on $(\dot{A}^{0,1}, \dot{\Phi})$.
The following proposition packages these into a single $\mathrm{sl}(E)$-valued equation on $(\dot{A}^{0,1}, \dot{\Phi})$. 

\begin{prop}\label{prop:coulombgauge}
 Let $(\Phi, A^{0,1})$ be a solution of the $SU(n)$-Hitchin's equations.
 A deformation $(\dot{\Phi}, \dot{A}^{0,1}) \in \Omega^{1,0} \oplus \Omega^{0,1}$ 
 is (1) in Coulomb gauge and (2) solves the infinitesimal version of $F_A+[\Phi, \Phi^{\dagger_{h_0}}]=0$ if, and only if,
 \begin{equation} \label{eq:ingauge}
  \del_A^{h_0} 
  \dot{A}^{0,1} + [\Phi^{\dagger_{h_0}}, \dot{\Phi}]=0.
 \end{equation}
\end{prop}

\begin{rem}
 In the case of the $4d$ anti-self-dual Yang-Mills equations (of which Hitchin's equations are a dimensional reduction), the linearization of $F_A^+=0$ and the Coulomb gauge condition can similarly be packaged into a single equation \cite[p. 55]{DonaldsonKronheimer}.
\end{rem}

\begin{proof}[Proof of Proposition \ref{prop:coulombgauge}]
The first thing to note is that because $\star \de z = -\I \de z$ and $\star \de \zbar = \I \de \zbar$, then the third equation in \eqref{eq:three} is
\begin{eqnarray} \label{eq:thirdequation}
0 &=&  \de_A \star \dot{A} - [\dot{\Phi} - \dot{\Phi}^{\dagger_{h_0}}, \star (\Phi - \Phi^{\dagger_{h_0}})]\\ \nonumber 
&=& \de_A \left(-\I \dot A^{1,0} + \I \dot A^{0,1} \right) +
\I [\dot{\Phi}, \Phi^{\dagger_{h_0}}]
-\I [\dot{\Phi}^{\dagger_{h_0}}, \Phi ]\\  \nonumber
&=& \I \left(\del^{h_0}_A \dot{A}^{0,1} - \delbar_A \dot{A}^{1,0} +
[\dot{\Phi}, \Phi^{\dagger_{h_0}}]
+ [\dot{\Phi}^{\dagger_{h_0}}, \Phi ]\right).
\end{eqnarray}

We compute the real and imaginary parts of \eqref{eq:ingauge},
where $2 \mathrm{Re}(X \de z \wedge \de \zbar) := (X + X^{\dagger_{h_0}}) \de z \wedge \de \zbar$. We extend the adjoint operator $\dagger_{h_0}$ to $2$-forms by  $(X \de z \wedge \de \zbar)^{\dagger_{h_0}} = X^{\dagger_{h_0}} \de \zbar \wedge \de z$.  Computing the real part of \eqref{eq:ingauge},
\begin{eqnarray}
 2\mathrm{Re}\left( \del_A^{h_0} 
  \dot{A}^{0,1} + [\Phi^{\dagger_{h_0}}, \dot{\Phi}] \right)&=&   \del_A^{h_0} 
  \dot{A}^{0,1} - (\del_A^{h_0} 
  \dot{A}^{0,1})^{\dagger_{h_0}} + [\Phi^{\dagger_{h_0}}, \dot{\Phi}] - [\Phi^{\dagger_{h_0}}, \dot{\Phi}]^{\dagger_{h_0}} \\ \nonumber
  &=&   \del_A^{h_0} 
  \dot{A}^{0,1} + \delbar_A 
  \dot{A}^{1,0} + [\dot{\Phi},\Phi^{\dagger_{h_0}}] + [\dot{\Phi}^{\dagger_{h_0}}, \Phi] 
  \\ \nonumber
  &=&  \de_A \dot{A} +[\Phi, \dot{\Phi}^{\dagger_{h_0}}] + [\dot{\Phi}, \Phi^{\dagger_{h_0}}]
\end{eqnarray}
Observe that the real part is equal to the second equation in \eqref{eq:three}.
Now computing the imaginary part of \eqref{eq:ingauge},
\begin{eqnarray}
2\I\mathrm{Im}\left( \del_A^{h_0} 
  \dot{A}^{0,1} + [\Phi^{\dagger_{h_0}}, \dot{\Phi}] \right)&=&   \del_A^{h_0} 
  \dot{A}^{0,1} + (\del_A^{h_0} 
  \dot{A}^{0,1})^{\dagger_{h_0}} + [\Phi^{\dagger_{h_0}}, \dot{\Phi}] + [\Phi^{\dagger_{h_0}}, \dot{\Phi}]^{\dagger_{h_0}} \\ \nonumber
  &=&   \del_A^{h_0} 
  \dot{A}^{0,1} - \delbar_A 
  \dot{A}^{1,0} + [\dot{\Phi},\Phi^{\dagger_{h_0}}] - [\dot{\Phi}^{\dagger_{h_0}}, \Phi] \\ \nonumber
  &\overset{\eqref{eq:thirdequation}}=&-\I \left(\de_A \star \dot{A} - [\dot{\Phi} - \dot{\Phi}^{\dagger_{h_0}}, \star (\Phi - \Phi^{\dagger_{h_0}})] \right).
\end{eqnarray}
Observe that the imaginary part is equal to the third equation in \eqref{eq:three}.
\end{proof}

\begin{rem}
Given a Higgs bundle deformation $(\dot{\eta}, \dot{\varphi})$, Proposition 
\ref{prop:coulombgauge} can also be viewed as a single equation for the $\mathfrak{sl}(E)$-valued deformation $\dot{\nu}$:
\begin{equation}\label{eq:ingaugetriple}
 \del_E^h \delbar_E \dot{\nu} - \del_E^h \dot{\eta} - \left[\varphi^{\dagger_h}, \dot{\varphi} + [\dot{\nu}, \varphi]\right]=0.
\end{equation}
This generalizes the observation in \cite[Eq. 4.14]{DumasNeitzke} 
 that for the $SU(2)$-Hitchin section, it is best to work with a single complex function $F$, which they call the ``complex variation.'' 
 In the usual local holomorphic frame for the Hitchin section
\begin{equation}
 \delbar_E =\delbar \qquad \varphi = \begin{pmatrix}0 & P \\ 1 & 0 \end{pmatrix} \de z, 
\end{equation}
these are locally related by $\dot{\nu} = -\frac{1}{2} F \begin{pmatrix}1 & 0 \\ 0 & -1 \end{pmatrix}$.
This metric variation $\dot{\nu}$ is essentially uniquely determined from $(\dot{\eta}, \dot{\varphi})$,
as we show in Corollary \ref{cor:dirichlet}.
\end{rem}

\begin{cor}\label{cor:dirichlet} 
Let $\Omega$ be a compact 
connected region with possibly-empty smooth boundary.
Given $(\dot{\eta}, \dot{\varphi})$ and Dirichlet data $\dot{\nu}_{\mathrm{fixed}}$ on $\del \Omega$, the deformation $\dot{\nu}$ solving the  Dirichlet problem in $\Omega$ for the induced PDE
\begin{eqnarray} \label{eq:Dirichletnu}
\del_E^h \delbar_E \dot{\nu} - \del_E^h \dot{\eta} - \left[\varphi^{\dagger_h}, \dot{\varphi} + [\dot{\nu}, \varphi]\right]&=&0\\ \nonumber 
 \dot{\nu} &=&  \dot{\nu}_{\mathrm{fixed}} \quad \mbox{in $\del\Omega$}
\end{eqnarray}
is unique in the Sobolev space $H^1_0(\Omega, \mathfrak{sl}(E))$ up to the addition of $\dot{\xi}$ satisfying
\begin{equation}\label{eq:vconditions}
 \delbar_E \dot{\xi} =0, \qquad [\varphi, \dot{\xi}]=0, \qquad \left.\dot \xi\right|_{\del \Omega}=0.
\end{equation}

\end{cor}

\begin{proof}[Proof of Corollary \ref{cor:dirichlet}]

 Suppose $\dot{\nu}_1$ and $\dot{\nu}_2$ both solve \eqref{eq:Dirichletnu}.
 Then $\dot{\xi} = \dot{\nu}_2-\dot{\nu}_1$ solves
\begin{eqnarray} 
\del_E^h \delbar_E \dot{\xi}  + \left[\varphi^{\dagger_h}, [\varphi, \dot{\xi}]\right]&=&0\\ \nonumber 
\left.\dot \xi\right|_{\del \Omega}&=&0.
\end{eqnarray}
Taking the inner product with $\dot\xi$, and integrating by parts (using the fact that $\dot\xi=0$ on $\del \Omega$),
\begin{eqnarray} \label{eq:ibp8}
0&=& \int_C \IP{\del_E^h \delbar_E \dot{\xi}, \dot\xi}_{h} +\IP{\left[\varphi^{\dagger_h}, [\varphi, \dot{\xi}]\right], \dot\xi}_{h_0}\\ \nonumber
&=& \int_C \norm{\delbar_E \dot\xi}^2_{h} + \norm{[\varphi, \dot\xi]}^2_{h}.
\end{eqnarray}
Therefore,
\begin{equation} \label{eq:nufreedom}
 \delbar_E \dot\xi =0 \qquad [\varphi, \dot\xi]=0,
\end{equation}
as claimed.
\end{proof}

\subsection{Expression for the hyperk\"ahler metric} \label{sec:hk}
Consequently, we can give a clean expression for Hitchin's hyperk\"ahler metric purely
in terms of $(\dot{\eta}, \dot{\varphi}, \dot{\nu})$.  Notice that the fixed hermitian metric 
$h_0$ from the unitary formulation of Hitchin's equations no longer appears!

\begin{prop} \label{prop:metricdiff}
Fix $(\dot{\eta}, \dot{\varphi}) \in T_{(\delbar_E, \varphi)} \cM'$. Let $\dot{\nu}$ be the associated metric deformation of $h$ solving the Coulomb gauge condition in \eqref{eq:ingaugetriple}.
Then 
Hitchin's hyperk\"ahler metric $g_{L^2}$ is 
\begin{equation} \label{eq:gL2triple}
 \norm{(\dot{\eta}, \dot{\varphi}, \dot{\nu})}_{g_{L^2}}^2
   = 2\int_C \IP{ \dot{\eta} - \delbar_E \dot{\nu},  \dot{\eta}}_h  +
 \IP{\dot{\varphi} +[\dot{\nu}, \varphi], \dot{\varphi}}_h.
\end{equation}
\end{prop}

\begin{proof}[Proof of Proposition \ref{prop:metricdiff}]
Given an infinitesimal deformation $(\dot{\eta}, \dot{\varphi}, \dot{\nu})$ solving \eqref{eq:Dirichletnu}, let  $(\dot{A}^{0,1}, \dot{\Phi})$ be the associated pair in unitary gauge solving the Coulomb gauge condition in \eqref{eq:ingauge}. Then, 
\begin{eqnarray} \label{eq:gL2comp}
\norm{(\dot{A}^{0,1}, \dot{\Phi})}^2_{g_{L^2}}&=&2\int_C \|\dot{A}^{0,1} \|^2_{h_0} + \|\dot{\Phi}\|^2_{h_0}\\ \nonumber 
&\overset{\eqref{eq:defun}}{=}&2\int_C \|H^{1/2} \left( \dot{\eta} - \delbar_E \dot{\nu} \right) H^{-1/2} \|^2_{h_0} + \|H^{1/2} \left( \dot{\varphi} +[\dot{\nu}, \varphi] \right) H^{-1/2}\|^2_{h_0}\\ \nonumber 
&=&2\int_C \|\dot{\eta} - \delbar_E \dot{\nu}\|^2_{h} + \|\dot{\varphi} +[\dot{\nu}, \varphi] \|^2_{h}\\ \nonumber 
  &\overset{IBP}{=}& 2\int_C  \IP{\dot{\eta} - \delbar_E \dot{\nu},\dot{\eta} }_{h}  -\IP{\del_E^{h}(\delbar_E \dot{\nu} -\dot{\eta}) - \left[\varphi^{\dagger_h}, \dot{\varphi} + [\dot{\nu}, \varphi]\right], \dot{\nu}}_h  \\ \nonumber
  & & +
  \IP{ \dot{\varphi} + [\dot{\nu}, \varphi], \dot{\varphi}}_{h}\\ \nonumber 
   &\overset{\eqref{eq:ingauge}}{=}&
     \int_C \IP{\dot{\eta} - \delbar_E \dot{\nu},  \dot{\eta} }_h  +
 \IP{ \dot{\varphi} +[\dot{\nu}, \varphi], \dot{\varphi}}_h.
\end{eqnarray}
In the integration-by-parts step we use the fact that the integrand $\IP{\dot{\eta}-\delbar_E \dot{\nu}, \dot{\nu}}_h$ is smooth, since $h$ is smooth on $C$.
\end{proof}

\begin{rem}
Given two variations $(\dot{\eta}_1, \dot{\varphi}_1, \dot{\nu}_1)$ and $(\dot{\eta}_2, \dot{\varphi}_2, \dot{\nu}_2)$ solving the Coulomb gauge condition in \eqref{eq:ingaugetriple}, then---as expected---
 \begin{eqnarray} \nonumber
\IP{(\dot{\eta}_1, \dot{\varphi}_1, \dot{\nu}_{1}) \;, \; 
(\dot{\eta}_2, \dot{\varphi}_2, \dot{\nu}_{2})}_{g_{L^2}}
&=&2 \mathrm{Re}
     \int_C \IP{\dot{\eta}_1 - \delbar_E \dot{\nu}_{1},  \dot{\eta}_2 }_{h}  +
 \IP{ \dot{\varphi}_1 +[\dot{\nu}_{1}, \varphi], \dot{\varphi}_2}_{h}.
 \end{eqnarray} 
\end{rem}

\subsection{The tangent space and metric on \texorpdfstring{$\cM_\infty'$ and $\cM_{\mathrm{app}}'$}{M}} \label{sec:othertangentspaces}
The formulation of the tangent space to $\cM'$ in \S\ref{sec:ingauge} can be extended to the 
related moduli spaces $\cM'_\infty$ and $\cM'_{\mathrm{app}}$.
The moduli space $\cM_\infty'$, introduced in \S\ref{sec:sfintro}, consists of triples $(\delbar_E, \varphi, h_\infty)$ up to gauge equivalence. The hermitian
metric $h_\infty$ solves the decoupled Hitchin's equations 
\begin{equation} \label{eq:decoupled}
 F_{D(\delbar_E, h_\infty)} =0,  \qquad [\varphi, \varphi^{\dagger_{h_\infty}}]=0.
\end{equation}
 A Higgs bundle deformation $(\dot{\eta}, \dot{\varphi})$ gives rise to a metric variation $\dot{\nu}_\infty$ such that 
\begin{eqnarray} \label{eq:defholinf} 
 (h_\infty)_\eps(w_1,w_2) &=& h_\infty( \e^{\eps \dot \nu_\infty} w_1,\e^{\eps \dot \nu_\infty} w_2).
\end{eqnarray}
In the unitary formulation of Hitchin's equations,
the variation $(\dot{\eta}, \dot{\varphi}, \dot{\nu}_\infty)$
corresponds to the deformation
\begin{equation} \label{eq:defuninf}
\dot{\Phi}_\infty = H_\infty^{1/2} \left( \dot{\varphi} +[\dot{\nu}_\infty, \varphi] \right)  H_\infty^{-1/2}
\qquad \dot{A}_\infty^{0,1} =H^{1/2}_\infty \left( \dot{\eta}  
- \delbar_{E_\infty} \dot{\nu}_\infty \right) H^{-1/2}_\infty, 
\end{equation}
where  $H_\infty$ is $\End E$-valued $h_0$-hermitian section such that $h_\infty(w_1,w_2)=h_0(H_\infty w_1, w_2)$.
Analogously to Proposition \ref{prop:coulombgauge},
given an infinitesimal Higgs bundle deformation $(\dot{\eta}, \dot{\varphi})$,
$\dot{\nu}_\infty$ solves the infinitesimal version of the decoupled Hitchin's equations
and is in formal Coulomb gauge if, and only if,
\begin{equation} \label{eq:7}
 \del_E^{h_\infty} \delbar_E \dot{\nu}_\infty - \del_E^{h_\infty} \dot{\eta} - \left[\varphi^{\dagger_{h_\infty}}, \dot{\varphi} + [\dot{\nu}_\infty, \varphi]\right]=0 
\end{equation}
and the infinitesimal version of $[\varphi, \varphi^{\dagger_{h_\infty}}]=0$ is satisfied:
\begin{equation}\label{eq:infinitesimaldecoupled}
 0=[\dot{\varphi}, \varphi^{\dagger_{h_\infty}}] + [\varphi, \dot{\varphi}^{\dagger_{h_\infty}}] + 
 \left[\varphi, 
 [\varphi^{\dagger_{h_\infty}}, \dot{\nu}_\infty + \dot{\nu}_\infty^{\dagger_{h_\infty}} ]
 \right].
\end{equation}
(Equivalently, the pair $(\dot{A}_\infty^{0,1}, \dot{\Phi}_\infty)$ solves 
\begin{eqnarray} \label{eq:ingaugeinf}
  0 &=& \del_{A_\infty}^{h_0} \dot{A}_\infty^{0,1} +  [\Phi_\infty^{\dagger_{h_0}}, \dot{\Phi}_\infty],
  \end{eqnarray}
  and 
  the infinitesimal version of $[\Phi_\infty, \Phi_\infty^{\dagger_{h_0}}]=0$.)

\bigskip

The moduli space $\cM_{\app}'$ is just the ordinary Higgs bundle moduli space of pairs $(\delbar_E, \varphi)$ upgraded to include the data of a hermitian metric $h^\app$ that is an  approximate solution of Hitchin's equations, as constructed in \cite{MSWW14, FredricksonSLn}.
The tangent space of $\cM_{\app}'$ consists of Higgs bundle deformations $(\dot{\eta}, \dot{\varphi})$ upgraded to contain
an infinitesimal deformation $\dot{\nu}^\app$ that 
does not infinitesimally change the (non-zero) value of $F_{D(\delbar_E, h^\app)} + [\varphi, \varphi^{\dagger_{h^\app}}]$. Thus, similarly, we can combine this with the formal Coulomb gauge condition into a single equation:
\begin{equation} \label{eq:7app}
 \del_E^{h_\app} \delbar_E \dot{\nu}_\app - \del_E^{h_\app} \dot{\eta} - \left[\varphi^{\dagger_{h_\app}}, \dot{\varphi} + [\dot{\nu}_\app, \varphi]\right]=0.
\end{equation}
This deformation corresponds to a pair 
$(\dot{A}^{0,1}_\app, \dot{\Phi}_\app) \in T_{(\delbar_E, \varphi, h^\app)} \cM$ 
solving 
\begin{eqnarray} \label{eq:ingaugeapp}
  0 &=& \del_{A_\app}^{h_0} \dot{A}_\app^{0,1} +  [\Phi_\app^{\dagger_{h_0}}, \dot{\Phi}_\app].
 \end{eqnarray}

\begin{cor} \label{cor:metricdiff}
Fix $(\dot{\eta}, \dot{\varphi}) \in T_{(\delbar_E, \varphi)} \cM'$. Let $\dot{\nu}_\app$ be the associated metric deformation of $h_{\app}$ solving the formal Coulomb gauge condition in \eqref{eq:7app}. 
Then the (non-hyperk\"ahler) approximate $L^2$-metric $g_{\app}$ on $\cM'_{\mathrm{app}}$ is 
\begin{equation} \label{eq:gapp}
 \norm{(\dot{\eta}, \dot{\varphi}, \dot{\nu}_\app)}_{g_{\app}}^2=2
 \int_C \IP{ \dot{\eta} - \delbar_E \dot{\nu}_\app,  \dot{\eta}}_{h_\app}  +
 \IP{\dot{\varphi} +[\dot{\nu}, \varphi], \dot{\varphi}}_{h_\app}.
\end{equation}
\end{cor}

\begin{proof}[Proof of Corollary \ref{cor:metricdiff}]
The proof for variations of $h^{\app}$ is the same as the proof of Proposition \ref{prop:metricdiff} because $h^{\app}$ is smooth on $C$.
\end{proof}

\begin{rem}
 We delay the proof of the analogous fact for $h_\infty$ until Theorem \ref{thm:semiflatisL2} since the integration-by-parts step can \emph{a priori} introduce boundary terms. This is because $h_\infty$ is singular at the branch locus $Z \subset C$,
 and consequently $\dot{\nu}_\infty$ may have singularities at the branch locus.
\end{rem}

\subsection{Semiflat metric for \texorpdfstring{$\cM'_{SU(n)}$}{M'}} \label{sec:semi}
In \cite{MSWW14}, Mazzeo-Swoboda-Weiss-Witt prove that the hyperk\"ahler semiflat metric $g_{\semif}$ on $\cM'$ is equal to the $L^2$-metric on $\cM'_\infty$ (denoted $g_{L^2(\cM'_\infty)}$).
Using 
the analogy of the clean expression for the hyperk\"ahler metric in \eqref{eq:ingaugetriple} for $g_{L^2}(\cM'_\infty)$, we give a proof that $g_{\semif}=g_{L^2}(\cM'_\infty)$ that works for \emph{any} $SU(n)$.

\subsubsection{Characterization of the semiflat metric} \label{sec:characterizationofsemiflat}
The map $\mathrm{Hit}: \cM' \rightarrow \cB'$ induces a hyperk\"ahler metric $g_{\semif}$
on $\cM'$, as follows. Since the base $\cB'$ is special K\"ahler, its cotangent bundle $T^* \cB'$ carries a canonical hyperk\"ahler
 metric \cite[Theorem 2.1]{Freedsemiflat}. 
 This descends to the quotient $\cM'\simeq T^* \cB'/\Gamma$  where $\Gamma$ is the local system with fiber $\Gamma_b$ \cite[Theorem 3.8]{Freedsemiflat}.
Our initial exposition follows \cite{MSWW17}.

The special K\"ahler metric on $\cB'$ comes from its interpretation as a family of spectral curves $\{\Sigma_{b}\}_{b \in \cB'}$.  A special K\"ahler structure can conveniently be  described
in terms of special conjugate holomorphic coordinates.
Choose a symplectic basis $\{\alpha_i(b), \beta_i(b)\}_{i=1}^{N}$ of the integer lattice $\Gamma_b=H_1(\Sigma_b,\Z)_{\sigma}$ of rank $N$. 
Here $H_1(\Sigma_b, \Z)_\sigma$ is a sublattice of $H_1(\Sigma_b, \Z)$ accounting
for the fact that the $SL(n,\C)$-Higgs field is traceless.  In particular, 
\begin{equation}\label{eq:sigma}
H_1(\Sigma_b, \Z)_\sigma = \ker \left( \pi: H_1(\Sigma_b, \Z) \rightarrow H_1(C, \Z) \right).
\end{equation}
 Then, these special conjugate
holomorphic coordinates are $z_i(b) =\int_{\alpha_i(b)} \lambda, w_i(b) =\int_{\beta_i(b)} \lambda$, defined by integrating the tautological $1$-form $\lambda$ over cycles in $\Sigma_b \subset \mathrm{Tot}(K_C)$. 
The special K\"ahler metric\footnote{This special K\"ahler metric is not unique; we have the freedom to rescale it by a constant.} on $\cB'$ is given by
 \begin{equation}
  \omega_{\mathrm{sK}} = - \sum_j \de z_j \wedge \de \overline{w}_j + \de \zbar_j \wedge \de w_j,
 \end{equation}
 By \cite{DouadyHubbard, Hertling}\footnote{The statement $\nabla^{GM}_{\dot{b}} \lambda = \dot \tau$ appears in \cite[Proposition 8.2]{Hertling}, where $\nabla^{GM}$ is the Gauss-Manin connection.  
 }and the discussion around \cite[Eq. 2]{MSWW17},
 \begin{equation}
  \de z_i(\dot{b}) =  \int_{\alpha_i(b)} \dot \tau(\dot{b}) \qquad \de w_i(\dot{b}) =  \int_{\beta_i(b)} \dot \tau(\dot{b}), 
 \end{equation}
using the the isomorphism
\begin{eqnarray} \label{eq:tau}
 \dot \tau: T_b \mathcal{B}' \simeq \sum_{i=2}^n H^0(C, K_C^i) &\rightarrow& H^0(\Sigma_b, K_{\Sigma_b})_\sigma \\ \nonumber
\dot{b} &\mapsto& \dot \tau(\dot{b}).
\end{eqnarray}
For example, in the $SL(2,\C)$ case, $\dot\tau: \dot{q_2} \mapsto \frac{\dot{q_2}}{2 \sqrt{q_2}}$.
(This is because one root of the characteristic polynomial $\lambda_\eps^2 - (q_2 + \eps \dot{q_2})$ is $\lambda_\eps=\sqrt{q_2} + \eps \frac{\dot{q_2}}{2 \sqrt{q_2}} + O(\eps^2)$.)
Consequently, (e.g. see \cite[\S2.3]{MSWW17}),
\begin{equation} \label{eq:sK}
 \norm{\dot{b}}^2_{g_{\mathrm{sK}}} = 2 \int_{\Sigma_b} |\dot \tau(\dot{b})|^2.
\end{equation}
\begin{rem}
The special K\"ahler metric is only defined up to scale, and correspondingly
there is a $\R^+$-family of special K\"ahler metrics on $\cB'$---and consequently semiflat metrics on $\cM'$---parameterized by a constant $\kappa>0$. For $\kappa=1$, the special K\"ahler metric is given by the expression in \eqref{eq:sK}.
\end{rem}
\bigskip

We now turn to a description of the semiflat metric on $\cM'$, concluding with practical
characterization of $g_{\semif}$ in Proposition \ref{prop:sfcharacterization}.
Given a spectral curve $\Sigma_b \in \cM'$, the fiber is $\mathrm{Hit}^{-1}(\Sigma_b) \simeq \mathrm{Prym}(\Sigma_b)$.
Note that
 \begin{equation}\mathrm{Prym}(\Sigma_b) \simeq H^0(\Sigma_b, K_{\Sigma_b})^*_{\sigma}/\Gamma_b,
 \end{equation}
 where $H^0(\Sigma_b, K_{\Sigma_b})_\sigma$ is a certain
 subspace (see \eqref{eq:sigma}) of $H^0(\Sigma_b, K_{\Sigma_b})$ accounting
 for the fact that $\mathrm{Prym}(\Sigma_b)$ is a subvariety of $\mathrm{Jac}(\Sigma_b)$.
 Consequently there is a canonical identification between the tangent space  to $\mathrm{Prym}(\Sigma_b)$ and the complex vector space $H^0(\Sigma_b, K_{\Sigma_b})^*_\sigma \simeq T_{b}^* \mathcal{B}'$.
Then for $(\cE, \varphi) \in \cM'$ with $\mathrm{Hit}(\cE, \varphi)=b$,
we have 
\begin{equation} \label{eq:tangentiso}
T_{(\cE, \varphi)} \cM' \simeq T_{(\cE, \varphi)} T^*\mathcal{B} \simeq  T_{b}^* \mathcal{B}' \oplus T_{b} \mathcal{B}'.
\end{equation}
The hyperk\"ahler metric $g_{\semif}$ on the tangent space to $(\cE, \varphi)$ is then defined to be 
$g_{\mathrm{sK}}^{-1} \oplus g_{\mathrm{sK}}$ where $g_{\mathrm{sK}}$ is the special K\"ahler metric on the $\cB'$ and $g_{\mathrm{sK}}^{-1}$ is the dual metric on the torus fiber.
Because $T_b \cB' \simeq H^0(\Sigma_b, K_{\Sigma_b})_\sigma$, 
the isomorphism in \eqref{eq:tangentiso} is
\begin{eqnarray} \label{eq:dualitysf}
 T_{(\delbar_E, \varphi)} \cM' &\rightarrow&  H^0(\Sigma_b, K_{\Sigma_b})^*_{\sigma}
 \oplus H^0(\Sigma_b, K_{\Sigma_b})_\sigma.
\end{eqnarray}
To express the semiflat metric on $\cM'$ concretely, we identify $H^0(\Sigma_b, K_{\Sigma_b})_\sigma$ with the space of horizontal deformations, and $H^0(\Sigma_b, K_{\Sigma_b})^*_{\sigma}$ with the space of vertical deformations, which we now define.
Starting here, our exposition departs from the perspective in \cite{MSWW17} and 
instead emphasizes spectral data.

\smallskip

\noindent \emph{Infinitesimal deformations of spectral data.} 
Given a Higgs bundle $(\delbar_E, \varphi) \in \cM'$, let $\cL \rightarrow \Sigma_b$ be the associated spectral data. As each point of the spectral cover $\Sigma_b$ over $p \in C$ represent an eigenvalue of $\left.\varphi\right|_{p}$, there is a canonical map
\begin{align}\tau: \Sigma_b &\rightarrow \pi^*K_C
\end{align}
 which associates the canonical eigenvalue of $\left.\pi^* \varphi\right|_{\widetilde{p}}$ to the point $\widetilde{p}$.  
The same complex vector bundle $L$ over the topological spectral curve
$\Sigma_{\mathrm{top}}$ underlies all nearby Higgs bundles.  The map $\tau$ gives an embedding of $\Sigma_{\mathrm{top}}$ into $\mathrm{Tot}(K_C)$.
Consequently, the spectral data is equivalent to the pair $(\delbar_L, \tau)$, since $\delbar_L$ gives $L$ a holomorphic structure and $\tau$ gives $\Sigma \subset \mathrm{Tot}(K_C)$ a complex structure from the complex structure of the ambient space.  
An infinitesimal deformation of the Higgs field gives rise to an infinitesimal variation  $\tau + \eps \dot \tau$ of the tautological eigenvalue $\pi^* \lambda$ of the pulled-back Higgs field $\pi^*\varphi$ on the spectral curve.
While the deformation $\dot \tau$ depends only on $\varphi$ and $\dot{\varphi}$,
the map from $(\dot{\eta}, \dot{\varphi})$ to $\dot{\xi}$ is more complicated to describe\footnote{Given a Higgs bundle $(\delbar_E, \varphi)$, let $\cL \rightarrow \Sigma_b$ be the associated spectral data.
To see the difficulty in passing from $(\dot{\eta}, \dot{\varphi})$ to $\dot{\xi}$, note that
the holomorphic line bundle $\cL$ arises from $(\delbar_E, \varphi)$ as follows: the underlying complex line bundle $L \subset \pi^*E$ is the eigenline bundle of the tautological eigenvalue of $\pi^* \varphi$ acting on $\pi^*E \rightarrow \Sigma_b$. The holomorphic structure on $\cL=(L, \delbar_L)$ is from the holomorphic structure of the ambient holomorphic bundle $\pi^*\cE \rightarrow \Sigma_b$. Now, given a deformation $(\dot{\eta}, \dot{\varphi})$ note that the section $L_\eps \subset \pi^*E$ moves in the family $\varphi +\eps \dot{\varphi}$, consequently, we first need to identify $L_\eps$ with a common fixed complex line bundle $L$. 
We can avoid this in the case where $\dot{\varphi}=0$. In this case, the deformation $\dot{\xi}$ is given in terms of $\dot{\eta}$ by $\pi^*\dot{\eta} \cdot s = \dot{\xi} s$, where $\pi^*\dot{\eta}$ is locally a $(0,1)$-form valued in $\mathfrak{sl}(n,\C)$ acting by matrix multiplication on any section $s$ of $L$, while $\dot{\xi}$ is locally a $(0,1)$-form acting by scalar multiplication on $s$.}.
Conveniently for our purposes, the existence of a correspondence is more important than the formula itself. 
To any abelianized associated deformation $(\dot{\xi}, \dot \tau)$, there is a canonically associated hermitian metric deformation $\dot{\nu}_L$ satisfying 
\begin{equation}\label{eq:dogs}
\del_L^{h_L} \delbar_L \dot{\nu}_L - \del^{h_L}_L \dot{\xi} =0.
\end{equation}
 This is the analog of \eqref{eq:7}. Looking at \eqref{eq:dualitysf}, note that $\dot{\tau} \in H^0(\Sigma_b, K_{\Sigma_b})_\sigma$. Looking forward to the proof of Proposition \ref{prop:sfcharacterization}, we'll identify $H^0(\Sigma_b, K_{\Sigma_b})_\sigma^*$ with $H^{0,1}(\Sigma_b)_{\sigma}$, and observe that $\dot \xi - \delbar_L \dot{\nu}_L$ is the harmonic representative of the class $[\dot \xi] \in H^{0,1}(\Sigma_b)_{\sigma}$.

\smallskip
\noindent \emph{Vertical deformations.} 
Vertical deformations are infinitesimal deformations in the torus fiber.
We identify the vector space of characteristic polynomials $\mathrm{char}_\varphi(\lambda)=\lambda^n - q_2 \lambda^{n-2} - \cdots - q_n$ with the vector space of coefficients $\cB \ni b=(q_2, \cdots, q_n)$.

\begin{defn}
A Higgs bundle deformation $(\dot{\eta}, \dot{\varphi})$ of $(\delbar_E, \varphi)$
is \emph{vertical} if $\dot{b}=0$, i.e. if $\left.\frac{\de}{\de \eps} \right|_{\eps=0} \mathrm{char}_{\varphi + \eps \dot{\varphi}}(\lambda)=0$. 
\end{defn}

An infinitesimal deformation of the Higgs field gives rise to an infinitesimal variation  $\tau+ \eps \dot\tau$ of the tautological eigenvalue of $\pi^*(\varphi + \eps \dot{\varphi})$

\begin{lem}\label{lem:vertical}
Let $(\delbar_E, \varphi)$ be a Higgs bundle in $\cM$ with $b \neq 0$.
Then a deformation $(\dot{\eta}, \dot{\varphi})$ is vertical, i.e. $\dot{b}=0$,
if and only if $\dot{\tau}$.
\end{lem}

\begin{proof}
The proof of this lemma is a standard argument concerning the passage between the roots of a polynomial and the coefficients of a polynomial.
 Deform the roots of a polynomial of degree $n$ by $\lambda_i + \eps \dot{\lambda_i}$ for $i=1, \cdots, n$. 
 Then the corresponding deformation of the polynomial is
 \begin{equation}
  \prod_{i=1}^n (x-\lambda_i) - \eps \sum_{i=1}^n \dot{\lambda}_i \prod_{k\in\{1, \ldots, \widehat{i}, \ldots, n\}} (x-\lambda_k) + O(\eps^2)
 \end{equation}
It is obvious that if $\dot{\lambda}_i$ all vanish then so does the polynomial deformation.

We want to show that if $ \sum_{i=1}^n \dot{\lambda}_i \prod_{k\in\{1, \ldots, \widehat{i}, \ldots, n\}} (x-\lambda_k)=0$ then $\lambda_i$ must all vanish, i.e. we must show that the polynomials  
\begin{equation}\label{eq:polys}
\left\{\prod_{k\in\{1, \ldots, \widehat{i}, \ldots, n\}} (x-\lambda_k)\right\}_{i=1}^n
\end{equation}
 are linearly independent. Indeed, this is true if and only if $\{\lambda_1, \ldots, \lambda_n\}$ are all distinct. (To see this, define the maps  $\mathrm{ev}_{\lambda_i}$ given by evaluation at $\lambda_i$ in the dual space of the space of polynomials of degree $n-1$. Then $\{\mathrm{ev}_{\lambda_1}, \ldots, \mathrm{ev}_{\lambda_n}\}$ are linearly independent elements provided $\{\lambda_1, \ldots, \lambda_n\}$ are all distinct. 
Each $\mathrm{ev}_{\lambda_i}$ sends the $i^{th}$ polynomial in \eqref{eq:polys} to a non-zero number and sends every other polynomial to zero. Thus, the polynomials are linearly independent.)

Now, we view $\lambda_i$ and $\dot{\lambda}_i$ as holomorphic sections rather than constants. The variation of the characteristic polynomial $\dot{b}$ is also holomorphic, i.e. each of the coefficients $q_i$ is a holomorphic section. Consequently, provided two or more sheets of the spectral cover do not globally coincide, we can conclude from $\dot{b}=0$ that $\dot{\lambda}_i=0$ for $i=1, \cdots, n$, i.e. $\dot \tau=0$. 
\end{proof}

\smallskip
\noindent \emph{Horizontal deformations.} 
The regular locus $\cM'$ of the Hitchin moduli space
admits a connection known as the Gauss-Manin connection\footnote{The Gauss-Manin connection is considerably more general.  Consider a family of algebraic  varieties 
$\mathcal{X} \rightarrow \mathcal{B}$ where the fibers $\pi^{-1}(b)=X_b$ are all diffeomorphic. Fixing $k \in \mathbb{N}$, the associated deRham cohomology groups $H^k_{dR}(X_b)$ are all isomorphic. The Gauss-Manin connection is a connection on the vector bundle $\mathcal{H} \rightarrow \mathcal{B}$ whose fiber over $b$ is $\mathcal{H}_b=H^k_{dR}(X_b)$. 
As such, the Gauss-Manin connection allows us to identify nearby fibers of $\mathcal{H}$.

For the $SU(n)$-Hitchin moduli space, the relevant family of varieties is the family of spectral curves $\Sigma_b$ and the relevant cohomology groups are $\mathcal{H}_b=H^{0,1}(\Sigma_b)_\sigma$,
since $T \mathrm{Prym}(\Sigma_b) \simeq H^0(\Sigma_b, K_{\Sigma_b})_\sigma^* \simeq H^{0,1}(\Sigma_b)_\sigma$. We view $\cM' \simeq \mathcal{H}/\Gamma$ where $\Gamma$ is a local system with fiber $\Gamma_b$.
}
which allows us to identify nearby fibers, and consequently define a horizontal deformation.
Note that each of the $2^{2 \cdot \mathrm{genus}(C})$ Hitchin sections are horizontal sections of $\mathrm{Hit}: \cM' \rightarrow \cB'$. 

The Gauss-Manin connection on the Hitchin moduli space is defined as follows:
As in the construction of the limiting metric $h_\infty$,
equip the spectral line bundle $\cL \rightarrow \Sigma_b$ with parabolic weights $-\frac{1}{2} \cdot \mathrm{ord}(\widetilde{p})$
at each of the ramification points $\widetilde{p} \in \Sigma$, where $\mathrm{ord}(\widetilde{p})$ is
an integer coming from the local order of the discriminant section which is described in \cite[\S2.1]{FredricksonSLn}. Note that for $\cM'_{SU(2)}$, $\mathrm{ord}(\widetilde{p})=1$.
There is a unique hermitian metric adapted to the parabolic weights solving 
$F_{D(\delbar_L, h_L)}=0$. This hermitian metric $h_L$ gives rise to two objects:
\begin{itemize}
 \item the limiting metric $h_\infty$ is obtained from orthogonal push-forward of $h_L$, and 
 \item $D(\delbar_L, h_L)$ is a projectively-flat $U(1)$-connection on the underlying complex vector bundle $L \rightarrow \Sigma_{\mathrm{top}}$ with monodromy $(-1)^{\mathrm{ord}(\widetilde{p})}$ around each of the ramification points $\widetilde{p} \in \Sigma$.
\end{itemize}
Since the same complex vector bundle $L$ over the topological spectral curve $\Sigma_{\mathrm{top}}$ underlies all nearby Higgs bundles, for any Higgs bundle near $[(\delbar_E, \varphi)]$, we also get a map $\pi_1(\Sigma-\widetilde{Z}) \rightarrow U(1)$.
We can identifying nearby
fibers of $\cM'$ through this topological data.  A deformation $(\dot{\eta}, \dot{\varphi})$ is said to be horizontal with respect to the Gauss-Manin connection---or simply \emph{horizontal}---if the holonomy
map $\pi_1(\Sigma-\widetilde{Z}) \rightarrow U(1)$
is infinitesimally constant in the direction $(\dot{\eta}, \dot{\varphi})$.
\medskip

\begin{lem}\label{lem:horizontal}
Fix a deformation $(\dot{\eta}, \dot{\varphi})$ of $(\delbar_E, \varphi)$.
Let $(\dot \tau, \dot{\xi})$ be the associated deformation of $(\lambda, \delbar_L)$,
and let $\dot{\nu}_L$ be the associated deformation of the Hermitian metric $h_L$ satisfying
\begin{equation}\label{eq:harmonic}
 \del_L^{h_L} \delbar_L \dot{\nu}_L - \del^{h_L}_L \dot{\xi} =0.
\end{equation}
A deformation is horizontal if, and only if,
\begin{equation}
 \delbar_L \dot{\nu}_L - \dot{\xi}  =0.
\end{equation}
\end{lem}

\begin{proof}
 Consider the deformation of $\delbar_L + \eps \dot{\xi}$ and $(h_L)_\eps(v,w) = h_L(\e^{\eps \dot{\nu}_L} v, \; \e^{\eps \dot{\nu}_L} w)$.
Then, in $h_L$-unitary local frame the corresponding deformation of the Chern connection is
 \begin{equation} \label{eq:defofCoulomb}
  D(\delbar_L, h_L) + \eps \left(\dot{\xi} -\delbar_L \dot{\nu}_L + (\dot{\xi} -\delbar_L \dot{\nu}_L)^{\dagger_{h_L}} \right) + O(\eps^2).
 \end{equation}
Given a horizontal deformation, the monodromy is constant, so for any loop $\alpha$ 
in the spectral curve $\Sigma$, the infinitesimal change in monodromy vanishes, i.e.
\begin{equation} \label{eq:monodromyinf}
  \int_\alpha \dot{\xi} -\delbar_L \dot{\nu}_L + (\dot{\xi} -\delbar_L \dot{\nu}_L)^{\dagger_{h_L}}  =0.
\end{equation}
Note that the integrand is a harmonic $1$-form \cite[Lemma 5.2.3]{Jost},
hence a representative for the cohomology $H^1(\Sigma, \R)$.
Because the integral vanishes over any cycle $\alpha$, the integrand
is identically zero, i.e. $\delbar_L \dot{\nu}_L - \dot{\xi}=0$, as claimed.
Conversely, if
$\delbar_L \dot{\nu}_L-\dot{\eta}=0$, then \eqref{eq:monodromyinf} is satisfied, so the deformation $(\dot \tau, \dot{\xi})$ is horizontal. 
\end{proof}

\bigskip

We can finally give a practical characterization of the semiflat metric 
 with scaling parameter $\kappa>0$ in terms of horizontal and vertical deformations. We use this in Theorem \ref{thm:semiflatisL2} to show that the $\kappa=1$ semiflat metric is the $L^2$-metric on $\cM'_\infty$.
\begin{prop} \label{prop:sfcharacterization}
The semiflat metric with scale $\kappa$ is characterized by three properties:
 \begin{enumerate}
 \item On horizontal deformations, the semiflat metric is $ \kappa \int_{\Sigma_b} 2|\dot \tau|^2$.
 \item On vertical deformations, the semiflat metric is $ \kappa^{-1} \int_{\Sigma_b} 2 | \dot{\xi} -\delbar_L \dot{\nu}_L|^2.$ 
 \item Horizontal and vertical deformations are orthogonal.
 \end{enumerate}
\end{prop}

\begin{proof}
 The first property is stated in \eqref{eq:sK}. The third property follows because $g_{\semif} = g_{\mathrm{sK}}^{-1} \oplus g_{\mathrm{sK}}$. For the second property, we need to unpack the duality in \eqref{eq:dualitysf}. (Note this appears in \cite[\S2.4.1]{MSWW17}.) Fiber-wise, the special K\"ahler metric $g_{\mathrm{sK}}$ is a metric on $V=H^0(\Sigma_b, K_{\Sigma_b})_\sigma$, the space of holomorphic $(1,0)$ forms on $\Sigma_b$. The metric $g_{\mathrm{sK}}$ defines a fiber-wise isomorphism $\Psi: V \rightarrow V^*=H^0(\Sigma_b, K_{\Sigma_b})^*_{\sigma} \simeq H^{1,0}(\Sigma_b)_\sigma$, as well as a
 metric $g_{\mathrm{sK}}^{-1}$ on the dual bundle which is fiber-wise defined on $V^*$ by $\Psi^* g_{\mathrm{sK}}^{-1}=g_{\mathrm{sK}}$. We can identify $V^*=H^0(\Sigma_b, K_{\Sigma_b})^*_\sigma \simeq H^{1,0}(\Sigma)^*_\sigma$ with $H^{0,1}(\Sigma_b)_\sigma$
by Serre duality, i.e. the natural pairing
$
  H^{0,1}(\Sigma_b) \times H^{1,0}(\Sigma_b) \rightarrow \C
$
 given by $(\alpha, \beta) \mapsto 2\int_{\Sigma_b} \alpha \wedge \beta$.
 Finally, we can identify $V^*$ with the space of harmonic forms $\mathcal{H}^{0,1}(\Sigma_b)_\sigma$. With these identifications, the map $\Psi^{-1}$ is simply
 \begin{eqnarray}
  \Psi^{-1}: V^* \simeq \mathcal{H}^{0,1}(\Sigma)_\sigma &\rightarrow& V \simeq H^{1,0}(\Sigma_b)_\sigma \\ \nonumber
\beta &\mapsto& \frac{1}{\kappa} \overline{\star} \beta.
 \end{eqnarray}
 Consequently, the metric on harmonic $(1,0)$-forms is simply given by $\norm{\beta}^2_{g_{\mathrm{sK}}^{-1}}=\kappa^{-1} \int_{\Sigma} 2 |\beta|^2$.
 Given a deformation $\dot{\xi}$ of $\delbar_L$ in $H^{0,1}(\Sigma_b)$, the associated harmonic representative is $\dot{\xi}-\delbar_L \dot{\nu}_L$ from \eqref{eq:harmonic}. Thus, the semiflat metric satisfies the second property.
\end{proof}

\subsubsection{Characterization of $L^2$-metric on \texorpdfstring{$\cM'_\infty$}{moduli space of limiting configurations}}

\begin{thm} \label{thm:semiflatisL2}
The $\kappa=1$ semiflat metric $g_{\semif}$ is the natural $L^2$-metric on the moduli space of limiting configurations $\cM'_\infty$, for deformations in formal Coulomb gauge.
Moreover, 
\begin{equation}\label{eq:gsf}
\norm{(\dot{\eta}, \dot{\varphi}, \dot{\nu}_{\infty})}^2 _{g_\semif}
=2
     \int_C \IP{\dot{\eta} - \delbar_E \dot{\nu}_{\infty},  \dot{\eta}}_{h_\infty}  +
 \IP{ \dot{\varphi} +[\dot{\nu}_{\infty}, \varphi], \dot{\varphi}}_{h_\infty}.
\end{equation}

\end{thm}
\begin{rem}
In the $SU(2)$ case, Mazzeo-Swoboda-Weiss-Witt prove that $g_{\semif}$ is the natural $L^2$-metric on the moduli space of limiting configurations in \cite[Proposition 3.7, Proposition 3.11, Lemma 3.12]
{MSWW17}. Comparing with \cite[Proposition 3.11]{MSWW17}, their harmonic representative $\alpha_\infty$ is the deformation of the Chern connection  $\left(\dot{\xi} -\delbar_L \dot{\nu}_L + (\dot{\xi} -\delbar_L \dot{\nu}_L)^{\dagger_{h_L}} \right)$ in \eqref{eq:defofCoulomb}.  
\end{rem}

\begin{proof}
We will prove that $g_{L^2}(\cM'_\infty)$ satisfies each of the three properties characterizing the semiflat metric in Proposition \ref{prop:sfcharacterization}.

As a corollary to the proof of Proposition \ref{prop:metricdiff}, the $L^2$-metric on $\cM_\infty'$ is given by 
\begin{eqnarray} \nonumber \label{eq:s1}
\norm{(\dot{\eta}, \dot{\varphi}, \dot{\nu}_{\infty})}^2_{g_{L^2}(\cM'_\infty)}
&=& \lim_{\eps \to 0} \left( 2 
     \int_{C_\eps} 
     \norm{\dot{A}_{\infty}^{0,1}}^2_{h_0} 
+ \norm{\dot{\Phi}}^2_{h_0} 
    \right)\\ \nonumber
    &=&
\lim_{\eps \to 0}\left( 2
     \int_{C_\eps} \IP{\dot{\eta} - \delbar_E \dot{\nu}_{1\infty},  \dot{\eta}}_{h_\infty}  +
 \IP{ \dot{\varphi} +[\dot{\nu}_{\infty}, \varphi], \dot{\varphi}}_{h_\infty}\right.\\ 
 & &  \quad -\left. 2 \mathrm{Re} \int_{\del C_\eps} \IP{\dot \eta- \delbar_E \dot{\nu}_{\infty}, \dot{\nu}_{\infty}}_{h_\infty} \right),
\end{eqnarray}
where $C_\eps$ is an arbitrary exhaustion of $C-Z$ by nested compact sets, such as $C_\eps=C-\cup_{p \in Z} B_\eps(p)$. Because $h_\infty$ is singular at the branch locus $Z$
and thus $\dot{\nu}_\infty$ may be singular, the integration-by-parts step generates the $\del C_\eps$ boundary terms. We first show that this boundary term vanishes.  To do this, we pullback and compute the integral on the spectral cover.
 Let $(\dot \tau, \dot{\xi}, \dot{\nu}_{L})$  be the corresponding deformation of the spectral data.
 Then, because everything diagonalizes on $\Sigma_b$, given any $1$-cycle $\gamma$ in $C$ with lift $\widetilde{\gamma}$ in $\Sigma_b$,
\begin{eqnarray}
 \int_{\gamma} \IP{\dot\eta- \delbar_E \dot{\nu}_{\infty}, \dot{\nu}_{\infty}}_{h_\infty}&=&
 \frac{1}{n} \int_{\widetilde{\gamma}} \IP{\pi^*  \dot \eta- \pi^*\delbar_E \pi^*\dot{\nu}_{\infty}, \pi^*\dot{\nu}_{\infty}}_{\pi^*h_\infty}\\ \nonumber 
 &=&
     \int_{\widetilde{\gamma}} \IP{\dot\xi - \delbar_L \dot{\nu}_{L},  \dot\nu_{L} }_{h_L}.
\end{eqnarray}
If $\widetilde{\gamma}$ is homotopic to the circle of radius $\eps$ around a ramification point $\widetilde{p} \in \widetilde{Z}$,
then because the monodromy of the associated abelian Chern connection 
is $(-1)^{\mathrm{ord}(\widetilde{p})}$, its infinitesimal variation vanishes.  Consequently, by a similar argument as in the proof of Lemma \ref{lem:horizontal},  $\int_{\widetilde{\gamma}} \dot{\xi} - \delbar_L \dot{\nu}_{L} + (\dot{\xi} - \delbar_L \dot{\nu}_{L})^\dagger =0$---for all variations, and not just horizontal variations.  Since this integrand is a harmonic representative for $H^1(\bD^\times, \R)$, 
the space of $1$-forms on the punctured local disk $\bD^\times$ around $\widetilde{p}$ by \eqref{eq:dogs}, the integrand vanishes, and $\dot{\xi} - \delbar_L \dot{\nu}_{L}=0$. Consequently, the boundary term in \eqref{eq:s1} vanishes. As a corollary, we've proved that $\norm{(\dot{\eta}, \dot{\varphi}, \dot{\nu}_{\infty})}^2 _{g_{L^2}(\cM'_\infty)}
=2
     \int_C \IP{\dot{\eta} - \delbar_E \dot{\nu}_{\infty},  \dot{\eta}}_{h_\infty}  +
 \IP{ \dot{\varphi} +[\dot{\nu}_{\infty}, \varphi], \dot{\varphi}}_{h_\infty}$, which appears in \eqref{eq:gsf}.

We now pullback the expression for $g_{L^2}(\cM'_\infty)$ in \eqref{eq:s1} to the spectral cover.
Because everything diagonalizes on $\Sigma_b$,
 \begin{eqnarray} \nonumber
  & &\IP{(\dot{\eta}_1, \dot{\varphi}_1, \dot{\nu}_{1, \infty}) \;, \; 
(\dot{\eta}_2, \dot{\varphi}_2, \dot{\nu}_{2, \infty})}_{g_{L^2}(\cM'_\infty)}\\ \nonumber 
&=&2 \mathrm{Re}
     \int_C \IP{\dot{\eta}_1 - \delbar_E \dot{\nu}_{1, \infty},  \dot{\eta}_2 }_{h_\infty}  +
 \IP{ \dot{\varphi}_1 +[\dot{\nu}_{1, \infty}, \varphi], \dot{\varphi}_2}_{h_\infty}\\ \nonumber  
 &=& \frac{2}{n}
  \mathrm{Re}
     \int_{\Sigma_b} \IP{\pi^* \dot{\eta}_1 - \pi^*\delbar_E \pi^*\dot{\nu}_{1, \infty},  \pi^*\dot{\eta}_2 }_{\pi^*h_\infty}  \\ \nonumber 
     & & \qquad+
 \IP{\pi^* \dot{\varphi}_1 +[\pi^*\dot{\nu}_{1, \infty}, \pi^*\varphi], \pi^*\dot{\varphi}_2}_{\pi^*h_\infty}\\ \nonumber 
 &=& 2
  \mathrm{Re}
     \int_{\Sigma_b} \IP{\dot\xi_1 - \delbar_L \dot{\nu}_{1,L},  \dot\xi_2 }_{h_L} +
 \IP{\dot\tau_1 +[\dot{\nu}_{1, L}, \lambda], \dot\tau_2}_{h_L}\\ \label{eq:usefulcharacterization}
  &=& 2
  \mathrm{Re}
     \int_{\Sigma_b} \IP{\dot\xi_1 - \delbar_L \dot{\nu}_{1,L},  \dot\xi_2 - \delbar_L \dot{\nu}_{2,L} } +
 \IP{\dot\tau_1, \dot\tau_2}
 \end{eqnarray}
 In the last line we used integration-by-parts and \eqref{eq:dogs}. 
We will use the expression in \eqref{eq:usefulcharacterization} to prove that $g_{L^2(\cM'_\infty)}$ satisfies the three characteristic properties of the semiflat metric
in Proposition \ref{prop:sfcharacterization}.

 \medskip

We prove that horizontal and vertical deformations are orthogonal in $g_{L^2}(\cM'_\infty)$.
 Let $(\dot\tau_1, \dot{\xi}_1, \dot{\nu}_{1, L})$ and 
 $(\dot\tau_2, \dot{\xi}_2, \dot{\nu}_{2, L})$ respectively be a horizontal and vertical deformation.  
 Then, because the horizontal deformations satisfies $\dot \xi_1 - \delbar_L \dot{\nu}_{1, L}=0$
 and vertical deformation satisfies $\dot\tau_2=0$, we see from \eqref{eq:usefulcharacterization} that $\IP{(\dot{\eta}_1, \dot{\varphi}_1, \dot{\nu}_{1, \infty}) \;, \; 
 (\dot{\eta}_2, \dot{\varphi}_2, \dot{\nu}_{2, \infty})}_{g_{L^2}(\cM'_\infty)}=0$.
 Thus $g_{L^2}(\cM'_\infty)$ satisfies the third characteristic property of the semiflat metric. It is immediate that for a horizontal deformation 
\begin{eqnarray}
 \norm{(\dot{\eta}, \dot{\varphi}, \dot{\nu}_\infty)}^2_{g_{L^2}(\cM')} 
 &=&  2 \int_{\Sigma_b} \left|\dot \tau \right|^2 ,
\end{eqnarray}
agreeing with the first characteristic property of the semiflat metric.
It is immediate that for a vertical deformation
\begin{eqnarray}
 \norm{(\dot{\eta}, \dot{\varphi}, \dot{\nu}_\infty)}^2_{g_{L^2}(\cM')} 
      &=& 2  \int_{\Sigma_b} |\dot\xi - \delbar_L \dot{\nu}_{L}|^2
\end{eqnarray}
Consequently, $g_{L^2}(\cM'_\infty)$ is the semiflat metric.
\end{proof}

\section{The asymptotic geometry of the \texorpdfstring{$SU(2)$}{SU(2)}-Hitchin moduli space}\label{sec:rank2}

We prove Theorem \ref{thm:main} for $SU(2)$ in \S\ref{sec:far}.  Our result depends on the local analysis in \S\ref{sec:near}.

\subsection{The difference \texorpdfstring{$g_{\mathrm{app}} - g_{\mathrm{sf}}$}{g\_app - g\_sf} on the local disks} \label{sec:near}
In this section, we study the difference $g_{\mathrm{app}} - g_{\mathrm{sf}}$ on local disks
around the branch points $Z$ of $\pi: \Sigma \rightarrow C$.

\begin{thm} \label{thm:local}
Suppose $q_2$ has a simple zero at $p \in C$. Let 
$z$ be a holomorphic coordinate on the disk such that $q_2=z \de z^2=-\det \varphi$. Additionally, take the local holomorphic frame
in which
 \begin{equation} \label{eq:simplezero}
  \delbar_E =\delbar \qquad \varphi = \begin{pmatrix}0 & z \\ 1 & 0 \end{pmatrix} \de z
  \qquad h_\infty = \begin{pmatrix} |z|^{-1/2} & 0\\ 0& |z|^{1/2} \end{pmatrix}.
 \end{equation}
Note that in this local holomorphic frame $h_t$ need not be diagonal; however,
as in \eqref{eq:htapprox}, define
 \begin{equation} \label{eq:175}
  h_t^\app = \begin{pmatrix} |z|^{-1/2} \e^{-u_t(|z|) \chi(|z|)} & 0\\ 0& |z|^{1/2} \e^{u_t(|z|) \chi(|z|)} \end{pmatrix}.
 \end{equation}
 
 \medskip
 
 Let $[(\dot{\eta}, \dot{\varphi})] \in T_{(\delbar_E, \varphi)} \cM'$ be an infinitesimal Higgs bundle deformation.
 Then on the disk $\bD$ around the point $p \in C$ there is a unique representative in the equivalence class where
 \begin{equation} \label{eq:defshape}
  \dot{\eta} =0, \qquad  \dot{\varphi} = \begin{pmatrix} 0 & \dot{P} \\ 0 & 0 \end{pmatrix} \de z, \qquad  \dot{\nu}_\infty = -\frac{\dot{P}}{4z} \begin{pmatrix}1& 0 \\ 0 & -1 \end{pmatrix} \qquad \mbox{for }\delbar \dot{P}=0.
 \end{equation}
This deformation is in formal Coulomb gauge.
In this local holomorphic frame, the deformation $\dot{\nu}_t^\app$ solving \eqref{eq:7app} is diagonal.
 
For the ray of variations $(0, t \dot{\varphi}) \in T_{(\delbar_E, t \varphi)} \cM'$, there is a constant $\gamma>0$ such that 
\begin{equation} \label{eq:propnear}
\norm{(0, t\dot{\varphi}, \dot{\nu}_t^\app)}^2_{g_{\app}(\bD)}-\norm{(0, t\dot{\varphi}, \dot{\nu}_\infty)}^2_{g_{\semif}(\bD)} =
O(\e^{-\gamma t}).
\end{equation}
\end{thm}

\begin{rem} \label{rem:holvar}
 The proof uses the holomorphic variations introduced in \cite{HubbardMasur}, and is a relatively straightforward adaptation of the clever argument in Dumas-Neitzke for $(\delbar_E, \varphi)$ in the Hitchin section,
 since
 our choice of local holomorphic frame in \eqref{eq:simplezero}
and representative of $[(\dot \eta, \dot \varphi)]$ in \eqref{eq:defshape} is the same as in \cite{DumasNeitzke}. We include the scaling factor $t$ to make the dependence on $t$ as explicit as possible. 

Note that Dumas-Neitzke do not use the approximate solution at all; we do because $h_t^{\app}$ is diagonal in the disk while the actual harmonic metric may not be diagonal.
 This has the advantage of reducing the analysis from a coupled system of PDEs to a single scalar PDE! For this, it is crucial that $h_\infty, \dot{\nu}_\infty, h_t^\app, \dot{\nu}_t^\app$ are diagonal.
\end{rem}

\begin{notate}
 For convenience, we will use the Pauli matrix notation $\sigma_3= \begin{pmatrix} 1 & 0 \\ 0 & -1 \end{pmatrix}$.
\end{notate}

\begin{proof}
\textsc{Claim 1:} \emph{There is a representative of the class $[(\dot{\eta}, \dot{\varphi})]$ as claimed in \eqref{eq:defshape}.}\\ $\triangleright$
From \eqref{eq:infinitesimal},
$[(\dot{\eta}, \dot{\varphi})]=[(\dot{\eta}', \dot{\varphi}')]$
if $\dot{\eta}'-\dot{\eta}=\delbar \dot \gamma$ and $\dot{\varphi}' -\dot{\varphi}= [\varphi,\dot \gamma]$ for infinitesimal complex gauge transformation $\dot \gamma$. Because $H^{0,1}(U)=0$, it is possible to choose $\dot{\gamma}$ such that $\delbar \dot{\gamma} = -\dot{\eta}$, thus $[(\dot{\eta}, \dot{\varphi})]=[(0, \dot{\varphi}')]$.
Now, taking a second infinitesimal gauge transformation $\dot{\gamma}$ defined entrywise by
$\dot{\gamma}_{11}=-\frac{1}{2} \dot{\varphi}_{21}'$ and $\dot{\gamma}_{12}=\dot \varphi_{11}'$ and $\dot{\gamma}_{21}=0$, we see that $[(\dot{\eta}, \dot{\varphi})]=[(0, \dot{\varphi}')]=[(0, \dot{\varphi}'')]$ for $\dot{\varphi}''$ strictly upper triangular
\begin{equation}
\dot{\varphi}'':=\dot{\varphi}'+ \left[\begin{pmatrix} 0 & z \\ 1 & 0 \end{pmatrix}, \dot{\gamma}\right] =
\begin{pmatrix} 0 & \dot{\varphi}_{12}' + \dot{\varphi}_{21}' z\\ 0 & 0 \end{pmatrix}.
\end{equation}
Note that $\dot{\eta}''=0$ because $\dot{\gamma}$ is holomorphic. Thus, we've proved the
claim.  

Note that there is some remaining infinitesimal gauge freedom.  Given $\dot{\gamma}$ holomorphic and satisfying $[\varphi, \dot{\gamma}]=0$, we have the freedom in \eqref{eq:infinitesimalnu} to shift $\dot{\nu}_\infty \mapsto \dot{\nu}_\infty + \dot{\gamma}$.
$\triangleleft$

\medskip

\noindent \textsc{Claim 2:} \emph{Furthermore, there is an infinitesimal gauge in which
the deformation $\dot{\nu}_\infty=-\frac{\dot{P}}{4z} \sigma_3$.}\\
$\triangleright$
Since $\dot{\nu}_\infty$ solves 
 \eqref{eq:infinitesimaldecoupled},
\begin{equation}
 \dot{\nu}_\infty +\dot{\nu}_\infty^{\dagger_{h_\infty}} = \tau \begin{pmatrix} 0 & z \\ 1 & 0 \end{pmatrix} -2 \mathrm{Re} \left(\frac{\dot{P}}{4z} \right) \begin{pmatrix} 1 & 0 \\ 0 & -1 \end{pmatrix},
\end{equation}
for some function $\psi$. Moreover, because $ (\dot{\nu}_\infty +\dot{\nu}_\infty^{\dagger_{h_\infty}})^{\dagger_{h_\infty}}=( \dot{\nu}_\infty +\dot{\nu}_\infty^{\dagger_{h_\infty}})$, $\psi = \overline{\psi} \frac{|z|}{z}$.
Since $\dot{\nu}_\infty$ solves \eqref{eq:7} and \eqref{eq:infinitesimaldecoupled},
$( \dot{\nu}_\infty +\dot{\nu}_\infty^{\dagger_{h_\infty}})$ also solves 
$\delbar_E \del_E^{h_\infty} (\dot{\nu}_\infty +\dot{\nu}_\infty^{\dagger_{h_\infty}} )=0$,
the infinitesimal version of $F_{D(\delbar_E, h_\infty)}=0$. It follows that 
$\psi$ is holomorphic. 
The infinitesimal gauge transformation $\dot{\gamma}=f \varphi_z$ shifts $\dot{\nu}_\infty + \dot{\nu}_\infty^{\dagger_{h_\infty}} \mapsto \dot{\nu}_\infty + \dot{\nu}_\infty^{\dagger_{h_\infty}} + \left(f + \overline{f}\frac{|z|}{z} \right) \varphi_z$; consequently, taking $\dot{\gamma} = - \frac{1}{2} \psi \varphi_z$, we reduce to the case where $\dot{\nu}_\infty +\dot{\nu}_\infty^\dagger =-2 \mathrm{Re}\left(\frac{\dot{P}}{4z} \right) \sigma_3$.
Lastly, by the Coulomb gauge condition in \eqref{eq:7}, we recover
$\dot{\nu}_\infty = - \frac{\dot{P}}{4z} \sigma_3.$ 
$\triangleleft$

\medskip

\noindent \textsc{Claim 3:} \emph{The deformation $\dot{\nu}_t^\app$ is diagonal.}\\
$\triangleright$ 
Since $h^{\app}_t$ is diagonal in the given local holomorphic frame, the equation  \eqref{eq:7app}  for $\dot{\nu}_t^\app$ decouples into equations for the diagonal and off-diagonal
parts of $\dot{\nu}_t^\app$. The deformation  $\dot{\nu}_t^\app$ agrees with $\dot \nu_\infty$ (which is diagonal) wherever $h_t^\app$ is equal to $h_\infty$, i.e. wherever the cutoff function $\chi =0$. Consequently, by the uniqueness of the solution of the Dirichlet problem (Corollary \ref{cor:dirichlet}) for the off-diagonal terms of $\dot{\nu}_t^\app$ on the closure of the disk $\bD_{\chi \neq 0}$, we conclude that $\dot{\nu}_\app$ is diagonal.
$\triangleleft$
\medskip

For convenience, we introduce the notation
\begin{equation} \label{eq:delta}
 \delta\left( (0, t\dot{\varphi}), (h_1, \dot \nu_1), (h_2, \dot \nu_2) \right)=
 2 \IP{\dot{\varphi} +[\dot{\nu}_1, \varphi], \dot{\varphi}}_{h_1}-
  2\IP{\dot{\varphi} +[\dot{\nu}_2, \varphi], \dot{\varphi}}_{h_2}.
\end{equation}
In order to prove the exponential decay of 
\begin{eqnarray*}
\norm{(0, t\dot{\varphi}, \dot{\nu}_t^\app)}^2_{g_{\app}(\bD)}-\norm{(0, t\dot{\varphi}, \dot{\nu}_\infty)}^2_{g_{\semif}(\bD)}&=& 2\int_C  
 \IP{\dot{\varphi} +[\dot{\nu}_t^\app, \varphi], \dot{\varphi}}_{h_t^\app}-
  \IP{\dot{\varphi} +[\dot{\nu}_\infty, \varphi], \dot{\varphi}}_{h_\infty}\\
  &=&  \int_C \delta \left( (0, t\dot{\varphi}), (h^{\mathrm{app}}_t, \dot \nu^{\mathrm{app}}_t), (h_\infty, \dot \nu_\infty) \right),
 \end{eqnarray*}
 we break the integrand 
into the following two pieces, and deal with these separately:
\begin{equation} \label{eq:decomp2}
\delta \left( (0, t\dot{\varphi}), (h^{\app}_t, \dot \nu^\app_t), (h^{\mathrm{model}}_t, \dot \nu^X_t) \right)
+
\delta \left( (0, t\dot{\varphi}), (h^{\mathrm{model}}_t, \dot \nu^X_t), (h_\infty, \dot \nu_\infty) \right).
\end{equation}
In the same local holomorphic frame as (\ref{eq:simplezero}, \ref{eq:175}), the metric $h_t^{\mathrm{model}}$ is (see \eqref{eq:htmodel})
\begin{equation} \label{eq:htmodel2}
   h_t^{\mathrm{model}} = \begin{pmatrix} |z|^{-1/2} \e^{-u_t(|z|)} & 0\\0 & |z|^{1/2} \e^{u_t(|z|)} \end{pmatrix}.
\end{equation}
The metric variation $\dot\nu^X_t$ is defined using a well-chosen holomorphic variation, as follows.
Suppose $\dot{P}= \sum_{n=0}^\infty a_n z^n$.
Then, closely following Dumas-Neitzke (see Eq. 10.12), 
let \begin{equation} \label{eq:chi}
  \chi = \sum_{n=0}^\infty \frac{a_n}{2n+1} z^n.
\end{equation}
in order to
define a holomorphic vector field $X=\chi \frac{\del}{\del z}$ generating the holomorphic deformation
well-suited to $\dot{q_2}=\dot{P} \de z^2$.
Define the complex function 
\begin{equation}
 F^X_t(z)=  \del_z\chi + 2 \chi \del_z(\frac{1}{2} \log |z| + u_t).
\end{equation}
Then take $\dot{\nu}^X_t = -\frac{1}{2} F^X_t \sigma_3$.
The variation $(0, t \dot{\varphi}, \dot{\nu}^X_t)$ of $(\delbar_E, t\varphi, h^{\mathrm{model}}_t)$
satisfies\footnotemark \eqref{eq:ingaugetriple}. (It is worth noting that $\dot{\nu}_\infty^X = -\frac{\dot{P}}{4z} \sigma_3 = \dot{\nu}_\infty$.)
\footnotetext{The fact that $\dot\nu^X_t$ satisfies the complex variation equation reduces to the fact that 
\begin{equation} \label{eq:scalarX}
 \left( \Delta - 16 t^2 |z| \cosh (2u_t)  \right) F^X_t + 8 t^2 \e^{-2u_t} |z|^{-1} \zbar \dot{P} =0.
\end{equation}
We can reduce this to the case $t=1$ which appears in \cite[Eq. 10.15]{DumasNeitzke} in which the deformation associated to $\dot{P}$ is 
\begin{equation} \label{eq:FX}
 F^X=  \del_z\chi + 2 \chi \del_z(\frac{1}{2} \log |z| + u_1),
\end{equation}
Define $\rho_t(z) =t^{-2/3} z$. One can check that 
$\rho_t^* u_t = u_1$. Because of this, by pulling back the expression \eqref{eq:scalarX} by $\rho_t$ and dividing by $t^{4/3}$, 
we get
\begin{equation}
\left( \Delta - 16 |z| \cosh (2u_1)  \right) \rho_t^*F^X_t + 8  \e^{-2u_1} |z|^{-1} \zbar t^{2/3} \rho_t^*\dot{P} =0.
\end{equation}
Thus the function $F_t^X$ in \eqref{eq:FX} is defined so that $\rho_t^* F^X_t$ solves the complex variation for the deformation $t^{2/3} \rho_t^*\dot{P}$.
}

\medskip

\noindent\textsc{Claim 4:} \emph{For the second piece of the integrand in \eqref{eq:decomp2}, there is a positive constant $\gamma>0$ such that
\begin{equation} \label{eq:claim2}
 \int_{\bD}\delta \left( (0, t\dot{\varphi}), (h^{\mathrm{model}}_t, \dot \nu^X_t), (h_\infty, \dot \nu_\infty) \right) = O( \e^{-\gamma t}).
\end{equation}
}
$\triangleright$
We prove this using Stokes' theorem.
By \cite[Lemma 12]{DumasNeitzke}, the integrand in \eqref{eq:claim2} is exact with
\begin{equation}
\delta \left( (0, t\dot{\varphi}), (h^{\mathrm{model}}_t, \dot \nu^X_t), (h_\infty, \dot \nu_\infty) \right)= \de \beta_t
\end{equation}
for 
\begin{equation} \beta_t = t^2 |z|^{-1}(\e^{-2u_t}-1) \left( 2 |z|^2 \star \de |\chi|^2 + |\chi|^2 \star \de |z|^2 \right).
\end{equation}
By 
Stokes' theorem, the integral in \eqref{eq:claim2} is $\int_{\del \bD} \beta_t$.
On the boundary of the disk $\bD$,  $|\chi|$ scales like $|\dot{P}|$.
The exponential decay of $\e^{-2u_1} -1$ in $|z|$ 
like $\e^{-\gamma |z|^{3/2}}$
in \cite{DumasNeitzke} implies\footnotemark the exponential decay of the term $\e^{-2u_t}-1$ as $\e^{-\gamma t}$.
\footnotetext{This is because \begin{equation}\label{eq:decay}\left.(\e^{-2u_t} -1)\right|_{|z|=R} =\left.\rho_t^*(\e^{-2u_t} -1)\right|_{|z|=t^{2/3} R} = \left.(\e^{-2u_1} -1)\right|_{|z|=t^{2/3} R}. \end{equation}}
This completes the proof of \eqref{eq:claim2}. $\triangleleft$

\medskip

\noindent \textsc{Claim 5:} \emph{ For the first piece of the integrand in \eqref{eq:decomp2}, there is a constant $\gamma>0$ such that 
\begin{equation} \label{eq:integrand4}
\int_{\bD} \delta \left( (0, t\dot{\varphi}), (h^{\app}_t, \dot \nu^\app_t), (h^{\mathrm{model}}_t, \dot \nu^X_t) \right) = O(  \e^{-\gamma t} ).
\end{equation}}
$\triangleright$ 
We further break the integrand in \eqref{eq:integrand4} into two pieces
\begin{equation} \label{eq:integrand3}
  \delta \left( (0, t\dot{\varphi}), (h^{\app}_t, \dot \nu^\app_t), (h^{\mathrm{model}}_t, \dot \nu^{\mathrm{model}}_t) \right) +  \delta \left( (0, t\dot{\varphi}), (h^{\mathrm{model}}_t, \dot \nu^{\mathrm{model}}_t), (h^{\mathrm{model}}_t, \dot \nu^X_t) \right),
\end{equation}
where $h_t^{\mathrm{model}}$ is as in \eqref{eq:htmodel2}, and the deformation $\dot{\nu}_t^{\mathrm{model}}$ solving \eqref{eq:ingaugetriple} is diagonal and agrees with $\dot{\nu}_t^\app$ on the disk where $\chi =1$. Dumas-Neitzke show that the integral of the the second piece of 
\eqref{eq:integrand3} decays like $O(\e^{-\gamma t})$ using the maximum principle.
The integral of the first piece of \eqref{eq:integrand3} reduces to the following integral on the annulus where $\chi \neq 1$, denoted $\bD-\bD_{\chi=1}$:
\begin{eqnarray}
& & \int_{\bD} \delta \left( (0, t\dot{\varphi}), (h^{\app}_t, \dot \nu^\app_t), (h^{\mathrm{model}}_t, \dot \nu^{\mathrm{model}}_t) \right)\\ \nonumber
 &\overset{\eqref{eq:delta}}{=}&
 \int_{\bD-\bD_{\chi=1}} |z|^{-1} \overline{\dot{P}} \left(
 (\e^{-2u_t}-\e^{-2u_t\chi})P 
 + 2z(\e^{-2u_t} F_t^{\mathrm{model}} - \e^{-2u_t \chi} F_t^{\app} )
 \right). 
\end{eqnarray}
Note that 
\begin{equation}
 \e^{-2u_t} F_t^{\mathrm{model}}-\e^{-2u_t \chi} F_t^\app=\left(\e^{-2u_t} - \e^{-2u_t \chi} \right) F_t^{\mathrm{model}} + \left(  F_t^{\mathrm{model}} -F_t^{\app} \right) \e^{-2u_t \chi}.
\end{equation}
Certainly $|z|, |z|^{-1}$, $|P|$, $|\dot{P}|$, $|F_t^{\mathrm{model}}|$, and $\e^{-2u_t \chi}$ are bounded on the annulus $|z| \in [R_1, R_2]$, because $\lim_{t \rightarrow \infty} F_t^{\mathrm{model}} = \frac{\dot{P}}{2z}$ and $\lim_{t \rightarrow \infty} \e^{-2u_t \chi}=1$. It suffices to show that for all points $x$ with $|x| \in [R_1, R_2]$ we have the following two bounds:
\medskip

\noindent \textsc{Subclaim 5A:} \emph{There exists constants $C,\gamma>0$ such that for all $\rho \in [R_1, R_2]$, $\left|\e^{-2u_t}-\e^{-2u_t\chi}|\right|_{\rho} < C\e^{-\gamma t} \left|e^{-2u_1}-\e^{-2u_1\chi}|\right|_{\rho})$.} 

\medskip

\noindent \textsc{Subclaim 5B:} $\int_{\bD-\bD_{\chi=1}} 2z|z|^{-1} \overline{\dot{P}} \e^{-2u_t \chi} \left(F_t^{\mathrm{model}} - F_t^{\app} \right) =O(\e^{-\gamma t})$

\medskip

\noindent$\triangleright\triangleright$ \emph{Proof of \textsc{Subclaim 5A}}: 
Note that  \begin{equation}\label{eq:decay5A}
\left.\left(\e^{-2u_t} -\e^{-2u_t\chi}\right)\right|_{|x|=\rho}<\left.\left(\e^{-2u_t} -1\right)\right|_{|x|=\rho} =\left.\rho_t^*(\e^{-2u_t} -1)\right|_{|x|=t^{2/3} \rho} = \left.(\e^{-2u_1} -1)\right|_{|x|=t^{2/3} \rho}. \end{equation}
From the exponential decay of $u_1$ in $|x|$ like $C\e^{-\gamma |x|^{3/2}}$ we get\footnotemark the exponential decay of this term like $C\e^{-\gamma t}$. $\triangleleft\triangleleft$ 
\footnotetext{
For a fixed value of $t$,
once we know that
\begin{equation}
 2u_t -2u_t \chi : \bD \rightarrow [0,C_t],
\end{equation}
we immediately get that
that for all $z \in \bD$
\begin{equation}
0 \leq \left(\e^{2u_t} - \e^{2u_t \chi} \right)\Big|_{z} \leq \e^{C_t}\left(2u_t- 2u_t \chi\right)\Big|_{z}.
\end{equation}
This is because the slope of any secant
line of the function $\e^x$ valued on $[0,C_t]$
is less than the slope of the tangent line to $\e^x$ at $x=C_t$.
}

\medskip

\noindent$\triangleright\triangleright$ \emph{Proof of \textsc{Subclaim 5B}}: 
First we note that at the outer edge of the annulus, $h_t^\app=h_\infty$, hence
\begin{equation}\label{eq:outer}
\left|F_t^{\mathrm{model}} - F_t^{\app}\right|_{|x|=R_2}
 =\left|F_t^{\mathrm{model}} - \frac{\dot{P}}{4z}\right|_{|x|=R_2} = O(\e^{-\gamma t}).
\end{equation}
The exponential decay follows from \cite[Theorem 8]{DumasNeitzke}.
Note that at the inner edge of the annulus $h_t^\app=h_t^{\mathrm{model}}$, hence
\begin{equation}\label{eq:inner}
\left|F_t^{\mathrm{model}} - F_t^{\app}\right|_{|x|=R_1}
 =0.
\end{equation}
Inside the annulus, we use the maximum principle.
For $|F_t^{\mathrm{model}} - F_t^{\app}|$, we note that (using the usual flat metric on the disk) 
$F_t^{\mathrm{model}}$ and $F_t^\app$ are respectively solutions of the complex variation equation \eqref{eq:ingauge} and \eqref{eq:7app}.  These reduce to the following scalar equations (cf. \eqref{eq:scalarX}) 
\begin{eqnarray}
 \left( \Delta - 16 t^2 |z| \cosh (2u_t  \,\,\,\,)  \right) F^{\mathrm{model}}_t \;+ 8 t^2 \e^{-2u_t \, \, \,} |z|^{-1} \zbar \dot{P} &=&0\\ \nonumber
  \left( \Delta - 16 t^2 |z| \cosh (2u_t \chi )  \right) F^{\app}_t + 8 t^2 \e^{-2u_t \chi } |z|^{-1} \zbar \dot{P} &=&0.
\end{eqnarray}
Thus,
 \begin{eqnarray} \label{eq:PDEmax}
0&=& \left( \Delta - 16 t^2 |z|\cosh(2u_t\chi) \right) (F_t^{\mathrm{model}}-F^{\app}_t)\\ \nonumber
 & &+16 t^2 |z| (\cosh(2u_t) - \cosh(2u_t  \chi)) F_t^{\app}  + 8t^2|z|^{-1} (\e^{-2u_t}-e^{-2u_t \chi}) \zbar \dot{P}.
\end{eqnarray}
We apply the maximum principle to the real and imaginary
parts of $(F_t^{\mathrm{model}}-F^{\app}_t)$ in \eqref{eq:PDEmax}.
For convenience, abbreviate these PDEs respectively $0=(\Delta - k^2) f_{\mathrm{Re}} + g_{\mathrm{Re}}$ and $0=(\Delta - k^2) f_{\mathrm{Im}} + g_{\mathrm{Im}}$.
Suppose $f_{\mathrm{Re}}$ does not achieve an interior maximum or minimum in the annulus.  Then because of the exponential decay of $|f_{\mathrm{Re}} + \I f_{\mathrm{Im}}|$ on the boundary of the annulus in (\ref{eq:outer},~\ref{eq:inner}), we have an exponentially-decaying upper bound on the interior of the annulus. 
Now, suppose $f_{\mathrm{Re}}$ achieves an interior maximum at $x_0$. Then $\Delta f_{\mathrm{Re}} \Big|_{x_0} \leq 0$. Thus, $f_{\mathrm{Re}}(x_0) \leq \frac{g_{\mathrm{Re}}}{k^2}\Big|_{x_0}$. Similarly, at an interior minimum, $f_{\mathrm{Re}}(x_0) \geq \frac{g_{\mathrm{Re}}}{k^2}\Big|_{x_0}$. Thus, if $f_{\mathrm{Re}}$ has an interior maximum or minimum at $x_0$, we see that $|f_{\mathrm{Re}}(x_0)| \leq \frac{g_{\mathrm{Re}}}{k^2}\Big|_{x_0}$.  Doing the same for $f_{\mathrm{Im}}$, we see that for any point $x$ in the annulus $|f|(x) \leq \max_{\bD - \bD_{\chi=1}} \frac{|g|}{k^2}$, i.e.
\begin{equation}
 \max |F_t^{\mathrm{model}} - F_t^{\app}| \leq \max \frac{16 (\cosh(2u_t)-\cosh(2u_t \chi) F_t^{\app} + 8 |z|^{-1} (\e^{-2u_t} - \e^{-2u_t \chi}) \zbar \dot{P}}{16|z| \cosh(2u_t \chi)}
\end{equation}
Note that we can bound this maximum value on the annulus by $O(\e^{-\gamma t})$ because $|\cosh(2u_t)-\cosh(2u_t \chi)|=O(\e^{-\gamma t})$ and $|\e^{-2u_t} - \e^{-2u_t \chi}| = O(\e^{-\gamma t})$.
Thus, since $z, |z|^{-1}, \overline{\dot{P}}, \e^{-2u_t \chi}$ are all uniformly bounded in $t$ on the annulus, we have the desired result:
\begin{equation}
 \int_{\bD - \bD_{\chi=1}} 2 z |z|^{-1} \overline{\dot{P}} \e^{-2u_t \chi}(F_t^{\mathrm{model}} - F_t^{\app}) = O(\e^{-\gamma t}).
\end{equation}
$\triangleleft\triangleleft$ 

\medskip

This completes the proof of \textsc{Claim 5}.
 $\triangleleft$
\end{proof}

\subsection{Main Theorem for \texorpdfstring{$SU(2)$}{SU(2)}} \label{sec:far}

\begin{thm}\label{7thm:mainsl2}
Fix a stable $SL(2,\C)$-Higgs bundle  $(\delbar_E, \varphi) \in \cM'$, and a Higgs bundle variation
$\dot{\psi}=(\dot{\eta}, \dot{\varphi})$. 
Consider the deformation $\dot{\psi}_t=(\dot{\eta}, t \dot{\varphi}) \in T_{(\delbar_E, t \varphi)} \cM$ over the ray $(\delbar_E, t \varphi, h_t)$.
As $t \rightarrow \infty$, the difference between 
 Hitchin's hyperk\"ahler $L^2$-metric $g_{L^2}$ on $\cM$ and the semiflat (hyperk\"ahler) metric $g_{\semif}$ is  exponentially-decaying.  In particular, there is some constant $\gamma>0$, such that 
 \begin{equation} \label{eq:summarysl2}
 \norm{(\dot{\eta}, t \dot{\varphi}, \dot{\nu}_t)}^2_{g_{L^2}}- \norm{(\dot{\eta}, t \dot{\varphi}, \dot{\nu}_\infty)}^2_{g_{\semif}}=  O(\e^{-\gamma t}),
 \end{equation}
 where the $g_{L^2}$ is defined in \eqref{eq:gL2triple} and $g_{\semif}$ is defined in \eqref{eq:gsf}.
\end{thm}

\begin{proof}
For notational simplicity, we suppress the metric variations and let
$g_{L^2}(\dot{\psi}_t, \dot{\psi}_t):=\norm{(\dot{\eta}, t \dot{\varphi}, \dot{\nu}_t)}^2_{g_{L^2}}$ and 
 $g_{\semif}(\dot{\psi}_t, \dot{\psi}_t):=\norm{(\dot{\eta}, t \dot{\varphi}, \dot{\nu}_\infty)}^2_{g_{\semif}}$.
We  
break the difference $ g_{L^2}(\dot{\psi}_t, \dot{\psi}_t) - g_{\semif}(\dot{\psi}_t, \dot{\psi}_t)$ into two pieces:
 \begin{eqnarray}\label{eq:twopieces}
   \left( g_{L^2}(\dot{\psi}_t, \dot{\psi}_t) - g_{\app}(\dot{\psi}_t, \dot{\psi}_t)  \right) +   \left( g_{\app}(\dot{\psi}_t, \dot{\psi}_t) - g_{\semif}(\dot{\psi}_t, \dot{\psi}_t)  \right).
 \end{eqnarray}
 By  \cite[\S10]{MSWW17}\footnotemark, there is a constant $\gamma$ such that 
\begin{equation}
     g_{L^2}(\dot{\psi}_t, \dot{\psi}_t) - g_{\app}(\dot{\psi}_t, \dot{\psi}_t)  =O(\e^{-\gamma t}).
\end{equation}
For the second piece of \eqref{eq:twopieces}, first note that Mazzeo-Swoboda-Weiss-Witt \cite{MSWW17} prove that $g_{\semif}$ is the $L^2$-metric on $\cM'_\infty$ (see Proposition \ref{prop:MSWW}).
For each $t \in \R^+$, the approximate metric $h_t^\app$ constructed in \cite{MSWW14} differs from $h_\infty$ only on disks around the branch points $Z$; consequently, the only contribution to the difference $ g_{\app}(\dot{\psi}_t, \dot{\psi}_t) - g_{\semif}(\dot{\psi}_t, \dot{\psi}_t)$ is from these disks. By Theorem \ref{thm:local}, 
\begin{equation}
  g_{\app}(\dot{\psi}_t, \dot{\psi}_t) - g_{\semif}(\dot{\psi}_t, \dot{\psi}_t) = O(\e^{-\gamma t}).
\end{equation}
\footnotetext{
To prove this, Mazzeo-Swoboda-Weiss-Witt use their earlier work
describing the family of harmonic metrics.
In
\cite{MSWW14},
they prove that
$h_t(w_1,w_2)= h_t^\app(\e^{-\kappa_t} w_1, \e^{-\kappa_t} w_2)$
for $h_0$-hermitian $\kappa_t$   satisfying
$\|\kappa_t\|_{H^2\left(i \mathfrak{su}(E)\right)} \leq C \e^{-\delta t}$.
(This formulation in terms of the hermitian metrics appears in \cite{FredricksonSLn}, but Mazzeo-Swoboda-Weiss-Witt's $h_0$-unitary formulation is equivalent to it.)
}
\end{proof}

\section{The asymptotic geometry of the \texorpdfstring{$SU(n)$}{SU(n)}-Hitchin moduli space}\label{sec:extension}

The above result extends to (most of) the $SU(n)$-Hitchin moduli space without much added difficulty.
Both hyperk\"ahler metrics $g_{L^2}$ and $g_{\semif}$ exist and are smooth on the regular locus $\cM' \subset \cM$. Though Gaiotto-Moore-Neitzke conjecture an asymptotic expansion of $g_{L^2}$ in terms of $g_{\semif}$  for any Higgs bundle in the regular locus $\cM'$, for technical reasons we restrict our attention to a slightly smaller space $\cM'' \subset \cM'$, the space
of Higgs bundles with simple eigenvalue crossing.
In the $SU(2)$ case, $\cM'=\cM''$.

\begin{defn}\cite{FredricksonSLn}\label{defn:simpleeigenvaluecrossing}
A Higgs bundle $(\delbar_E, \varphi)$ has \emph{simple eigenvalue crossing} if the resultant of the characteristic polynomial of $\varphi$, $\mathrm{char}_\varphi(\lambda)$, and its derivative $\del_\lambda \mathrm{char}_\varphi (\lambda)$ does not vanish on $C$.
\end{defn}

\begin{rem} \label{rem:simpleeigenvaluecrossing} We comment on the reason for restricting to $\cM''$.
The spectral curve $\pi: \Sigma \to C$ is ramified at the branch divisor $Z$, the set of zeros of
the discriminant section
\begin{eqnarray} \label{eq:discriminant}
\Delta_\varphi: C &\rightarrow& K_C^{n^2-n}\\ \nonumber
p &\mapsto& \prod_{1 \leq i<j \leq n} (\lambda_i(p) - \lambda_j(p))^2,
\end{eqnarray}
where $\lambda_1, \cdots, \lambda_n$ are the eigenvalues of $\varphi$.
A Higgs bundle is \emph{discriminant-simple} if all the zeros of $\Delta_\varphi$ are simple. Slightly more generally, 
\begin{figure}[h]
\begin{centering}
\includegraphics[height=1in]{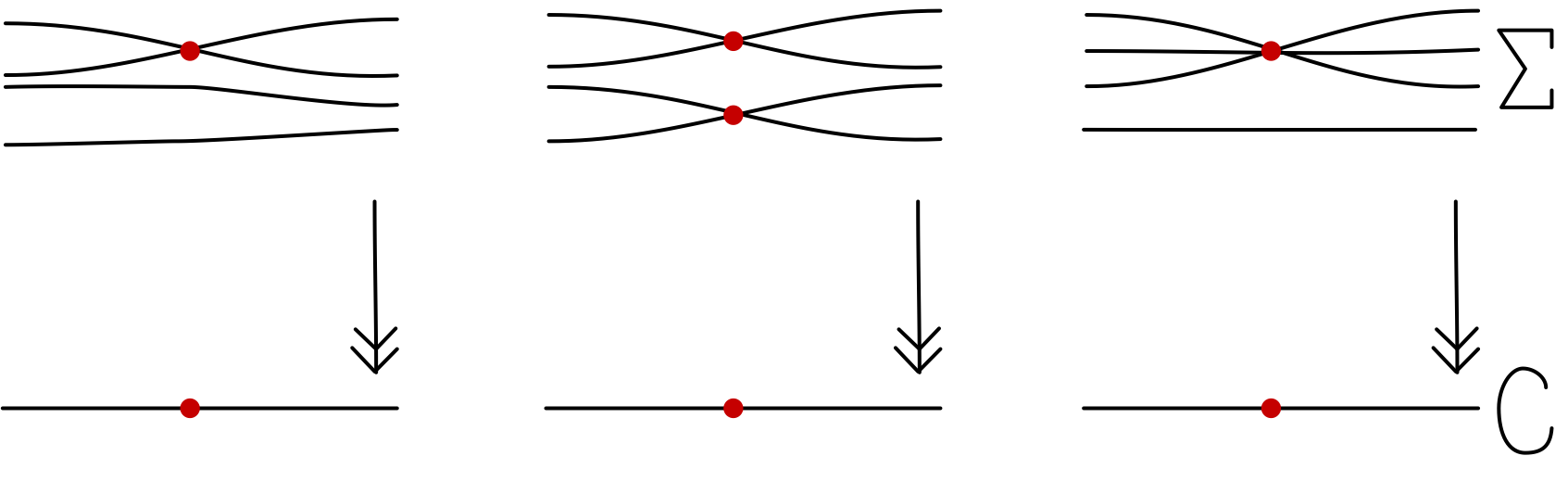}
\caption{\label{fig:sheets} Local spectral covers for 
(\textsc{Left}) a discriminant-simple Higgs bundle;
(\textsc{Middle}) a Higgs bundle with simple eigenvalue crossing;
(\textsc{Right}) a regular Higgs bundle, here featuring the local model $\lambda \mapsto \lambda^3 =z \de z^3$.
}
\end{centering}
\end{figure}
if a Higgs bundle has simple eigenvalue crossing, then near a branch point, pairs of eigenvalues may ramify, as shown in Figure \ref{fig:sheets}.
However, triples, etc. of eigenvalues may not ramify. 
(See the discussion in \cite{FredricksonSLn}.)
The regular locus $\cM'$ of the $SU(n)$-Hitchin moduli space is stratified by the types of ramification\cite[\S2.3]{FredricksonSLn}, and $\cM''$ is the lowest strata of the regular locus.   

The author described the harmonic metrics for Higgs bundles near the ends of the  $SU(n)$-Hitchin moduli in \cite{FredricksonSLn}. Near each branch point in $\cM''$, there is a local holomorphic frame in which \eqref{eq:simplezerosln} holds. In each of the $2 \times 2$ blocks, we can use the local analysis of \S\ref{sec:near} for the $2 \times 2$ local model 
\begin{equation}
\delbar_E=\delbar \quad \varphi = \begin{pmatrix} 0  & z \\ 1  & 0  \end{pmatrix} \de z \qquad
\dot{\eta}=0 \quad \dot{\varphi} = \begin{pmatrix} 0 & \dot{P} \\ 0 &0  \end{pmatrix} \de z.
\end{equation}
We restrict from the regular locus $\cM'$ to the locus $\cM''$ since this $2 \times 2$ local model is sufficient there, and we don't need higher-order local models such as the $3 \times 3$ local model
\begin{equation}
\delbar_E=\delbar \quad \varphi = \begin{pmatrix} 0 & 0 & z \\ 1 &0 & 0 \\ 0 & 1 & 0 \end{pmatrix} \de z \qquad
\dot{\eta}=0 \quad \dot{\varphi} = \begin{pmatrix} 0 & \dot{P}_1 & \dot{P}_2 \\ 0 &0 & \dot{P}_1 \\ 0 & 0 & 0 \end{pmatrix} \de z.
\end{equation}

These higher-order local models require new techniques. To see this in the $3 \times 3$ local model, note that the associated discriminant section $\Delta_\varphi$ has a double zero at $z=0$. The variation $\dot{P}_2$ moves the double zero, while the variation $\dot{P}_1$ decomposes the double zero of $\Delta_{\varphi}$ into two simple zeros.  One can use Dumas-Neitzke's technique of an infinitesimal biholomorphic flow to deal with the variation $\dot{P}_2$. However, a new technique is needed for $\dot{P}_1$, since we infinitesimally pass from a higher strata to a lower strata in $\cM'$.  In this, it is worth noting that the author's construction of approximate harmonics for Higgs bundles near the ends of the $SU(n)$-Hitchin moduli space only works uniformly in each strata\cite[Remark 6.5]{FredricksonSLn}; consequently, a first step towards dealing with variations like $\dot{P}_1$ might be a uniform description of the ends of the $SU(n)$-Hitchin moduli space.
\end{rem}

\begin{thm}\label{7thm:mainsln}
Fix a polystable $SL(n,\C)$-Higgs bundle  $(\delbar_E, \varphi) \in \cM''$, and a Higgs bundle variation
$\dot{\psi}=(\dot{\eta}, \dot{\varphi})$. 
Consider the deformation $\dot{\psi}_t=(\dot{\eta}, t \dot{\varphi}) \in T_{(\delbar_E, t \varphi)} \cM$ over the ray $(\delbar_E, t \varphi, h_t)$.
As $t \rightarrow \infty$, the difference between 
 Hitchin's hyperk\"ahler $L^2$-metric $g_{L^2}$ on $\cM$ and the semiflat (hyperk\"ahler) metric $g_{\semif}$ is  exponentially-decaying.  In particular, there is some constant $\gamma>0$, such that 
 \begin{equation} \label{eq:summarysln}
\norm{(\dot{\eta}, t \dot{\varphi}, \dot{\nu}_t)}^2_{g_{L^2}} -\norm{(\dot{\eta}, t \dot{\varphi}, \dot{\nu}_\infty)}^2_{g_{\semif}} =  O(\e^{-\gamma t}).
 \end{equation}
\end{thm}

The two key results needed for the proof of Theorem  \ref{7thm:mainsln} are:
\begin{itemize}
 \item Theorem \ref{thm:semiflatisL2}, which states that the semiflat metric is the $L^2$-metric on the moduli space of limiting configurations $\cM''_\infty$ is the semiflat metric.  (Note that Mazzeo-Swoboda-Weiss-Witt's proof in the $SU(2)$ case does not extend to the $SU(n)$ case.)
 \item Proposition \ref{thm:localSLn} which is the $SU(n)$ version of Theorem \ref{thm:local}.
\end{itemize}
We first state and prove Proposition \ref{thm:localSLn}, then prove Theorem \ref{7thm:mainsln}.

\begin{prop} \label{thm:localSLn}
Fix a $SL(n,\C)$-Higgs bundle with simple eigenvalue crossing $(\delbar_E, \varphi) \in \cM''$. Given a branch point $p \in Z \subset C$ at which the discriminant section $\Delta_\varphi$ in \eqref{eq:discriminant} has a zero of order $\ell$, we work in the local holomorphic frame and local coordinates $z_1, \cdots, z_\ell$ on the disk $\mathbb{D}$ centered at $p$ in which 
 \begin{equation} \label{eq:simplezerosln}
  \delbar_E =\delbar \quad \varphi = \bigoplus_{j=1}^\ell\begin{pmatrix}\widehat{\lambda}_{(j)} & z_j \de z_j \\ \de z_j& \widehat{\lambda}_{(j)} \end{pmatrix} \oplus \begin{pmatrix} \lambda_{n-2\ell+1} && \\ & \ddots & \\ & & \lambda_n \end{pmatrix}\quad h_\infty = \bigoplus_{j=1}^\ell\begin{pmatrix} |z_j|^{-1/2} & 0\\ 0& |z_j|^{1/2} \end{pmatrix} \oplus \mathrm{Id}_{n-2\ell}.
 \end{equation}
 (The existence of such a local holomorphic frame is proved in \cite[Proposition 3.5]{FredricksonSLn}.
 Note that in this local frame $h_t$ need not be diagonal; however,  
in this local frame the family of approximate hermitian metric defined in \cite[Definition 4.9]{FredricksonSLn} by
 \begin{equation}\label{eq:htsln}
  h_t^\app = 
  \bigoplus_{j=1}^\ell\begin{pmatrix} |z_j|^{-1/2}\e^{-u_t(|z_j|) \chi(|z_j|)} &0 \\ 0& |z_j|^{1/2} \e^{u_t(|z_j|) \chi(|z_j|)}\end{pmatrix} \oplus \mathbf{1}_{n-2\ell}.
 \end{equation}
 are all diagonal. The
 function $u_t$ is the same function that appeared in the rank $2$ case in \eqref{eq:umodel}.
)

 \medskip
 Let $[(\dot{\eta}, \dot{\varphi})] \in T_{(\delbar_E, \varphi)} \cM'$ be an infinitesimal Higgs bundle deformation.
 Then on the disk $\mathbb{D}$ around the point $p \in C$, there is a unique representative
 in the equivalence class in which
 \begin{eqnarray} \label{eq:defshapesln}
  \dot{\eta} &=&0\\ \nonumber 
  \dot{\varphi} &=& \bigoplus_{j=1}^\ell\begin{pmatrix} \dot{\lambda}_{(j)} & \widehat{\dot{P}}_{(j)} \de z_j\\ 0 & \dot{\lambda}_{(j)} \end{pmatrix}   \oplus \begin{pmatrix} \dot{\lambda}_{2n-\ell + 1} & & \\ & \ddots & \\ & & \dot{\lambda}_{2n} \end{pmatrix} \\ \nonumber 
  \dot{\nu}_\infty &=& \bigoplus_{j=1}^\ell 
  -\frac{\widehat{\dot{P}}_{(j)}}{4z_j} \begin{pmatrix}1& 0 \\ 0 & -1 \end{pmatrix} 
  \oplus \mathbf{0}_{n-2\ell}
 \end{eqnarray}
 
 \medskip
For the ray of variations $(0, t \dot{\varphi}) \in T_{(\delbar_E, t \varphi)} \cB''$, there is a constant $\gamma>0$ such that 
\begin{equation} \label{eq:propnearSLn}
\norm{(0, t\dot{\varphi}, \dot{\nu}_t^\app)}^2_{g_{\app}(\bD)}-\norm{(0, t\dot{\varphi}, \dot{\nu}_\infty)}^2_{g_{\semif}(\bD)} =
O(\e^{-\gamma t}).
\end{equation}
\end{prop}

\begin{proof}[Proof of Proposition \ref{thm:localSLn}]\hfill

\noindent\textsc{Claim 1:} \emph{There is a representative of the class $[(\dot{\eta}, \dot{\varphi}, \dot{\nu}_\infty)]$ as claimed in \eqref{eq:defshapesln}.}\\ $\triangleright$
Recall from \eqref{eq:infinitesimal} and \eqref{eq:infinitesimalnu} that the infinitesimal gauge transformation $\dot{\gamma}$ acts by
\begin{eqnarray}
 \dot{\eta} &\mapsto& \dot{\eta} + \delbar_E \dot{\gamma}\\ \nonumber 
 \dot{\varphi} &\mapsto& \dot{\varphi} +[\varphi, \dot{\gamma}]\\ \nonumber
 \dot{\nu} & \mapsto & \dot{\nu} + \dot{\gamma}.
\end{eqnarray}
Since $H^{0,1}(\mathbb{D})=0$, we can find $\dot{\gamma}$ such that $\dot{\eta} +\delbar \dot{\gamma}=0$. Consequently, there is a representative with $\dot{\eta}=0$. 

We now seek to produce a holomorphic $\dot{\gamma}$ such that $\dot{\varphi}+[\varphi, \dot{\gamma}]$ is as claimed in \eqref{eq:defshapesln}.
Define
\begin{equation}
 g_{\diamond}= \oplus_{i=1}^\ell\begin{pmatrix} z_i^{1/4} &0 \\ 0& z_i^{-1/4} \end{pmatrix} \begin{pmatrix} \frac{1}{\sqrt{2}} & \frac{1}{\sqrt{2}} \\ - \frac{1}{\sqrt{2}} & \frac{1}{\sqrt{2}} \end{pmatrix} \oplus \mathbf{1}_{n-2\ell},
\end{equation} 
and observe that 
\begin{eqnarray}
\widetilde{\varphi}=
  g^{-1}_{\diamond} \varphi g_{\diamond}&=&   \bigoplus_{j=1}^\ell\begin{pmatrix}\widehat{\lambda}_{(j)}- \sqrt{z_j} \de z_j & 0 \\0 & \widehat{\lambda}_{(j)} + \sqrt{z_j} \de z_j \end{pmatrix} \oplus \begin{pmatrix} \lambda_{n-2\ell+1} && \\ & \ddots & \\ & & \lambda_n \end{pmatrix} 
\end{eqnarray}
is diagonal.
Given a Higgs bundle deformation $\dot{\varphi}$, and a pair $(i,j)$ off the block diagonal,   $\widetilde{\dot{\gamma}}_{ij}$ is determined by the equation  
\begin{equation}
 \left(g^{-1}_{\diamond} \dot{\varphi} g_{\diamond} +[\widetilde{\dot{\varphi}}, \widetilde{\dot{\gamma}}]\right)_{ij}=0.
\end{equation}
Here, we use that $[\widetilde{\dot{\varphi}}, \widetilde{\dot{\gamma}}]_{ij}
=(\widetilde{\dot{\varphi}}_{ii}-\widetilde{\dot{\varphi}}_{jj})\widetilde{\dot{\gamma}}_{ij}$ and the fact that $(\widetilde{\dot{\varphi}}_{ii}-\widetilde{\dot{\varphi}}_{jj})$
is non-zero on $\mathbb{D}$.
In each of the $2\times 2$ blocks
where 
\begin{equation}
\varphi_{2 \times 2} = \begin{pmatrix} \widehat{\lambda}_{(j)} & z \de z \\ \de z & \widehat{\lambda}_{(j)} \end{pmatrix}, 
\qquad \dot{\varphi}_{2 \times 2}= \begin{pmatrix} \dot P + \dot P_1 & \dot P_2 \\ \dot P_3 & \dot P - \dot P_1 \end{pmatrix} \de z,
\end{equation}
define the corresponding $2 \times 2$ block of $\widetilde{\dot\gamma}$ to be
\begin{equation}
\widetilde{\dot{\gamma}}_{2 \times 2} = 
\left(g^{-1}_\diamond \dot{\gamma} g \right)_{2 \times 2} = 
\begin{pmatrix} -\frac{\dot P_1}{2 \sqrt{z}} & -\frac{\dot P_3}{2} + \frac{\dot P_1}{2 \sqrt{z}}  \\
 -\frac{\dot P_3}{2} - \frac{\dot P_1}{2 \sqrt{z}} & \frac{\dot P_1}{2 \sqrt{z}}                            
                                      \end{pmatrix}.
\end{equation}
For $i > 2\ell$, define $\widetilde{\dot{\gamma}}_{ii}=0$.
Having defined $\widetilde{\dot{\gamma}}$, define
\begin{equation}\label{eq:defofgamma}
 \dot{\gamma} = g_{\diamond} \widetilde{\dot{\gamma}} g_{\diamond}^{-1}.
\end{equation}
By construction, formally
\begin{equation}
 \dot{\varphi} + [\varphi, \dot{\gamma}] = \bigoplus_{j=1}^\ell\begin{pmatrix} \dot{\lambda}_{(j)} & \widehat{\dot{P}}_{(j)} \de z_j\\ 0 & \dot{\lambda}_{(j)} \end{pmatrix}   \oplus \begin{pmatrix} \dot{\lambda}_{2n-\ell + 1} & & \\ & \ddots & \\ & & \dot{\lambda}_{2n} \end{pmatrix}.
\end{equation}
It remains to check that the variation $\dot{\gamma}$ defined by \eqref{eq:defofgamma} (which \emph{a priori} may have negative powers and roots of $z_i$ appearing) is truly holomorphic. We show that $\dot{\gamma}$ is holomorphic by computing it explicitly in terms of $\dot{\varphi}$ in each of the following $6$ types of blocks, which can be seen in this $n=6$ example:
\begin{equation} \label{eq:colored}
\varphi:=
\left(\begin{array}{@{}c|c|c |c}
  \begin{matrix}
  \widehat{l}_1 & z \\
  1 &  \widehat{l}_1
  \end{matrix}
  & \color{blue}{\begin{matrix} 0 & 0 \\ 0 & 0 \end{matrix}} & \color{purple}{\begin{matrix} 0  \\ 0 \end{matrix}} & \color{purple}{\begin{matrix} 0  \\ 0 \end{matrix}} \\
\hline

 \color{blue}{\begin{matrix} 0 & 0 \\ 0 & 0 \end{matrix}}&
  \begin{matrix}
  \widehat{l}_{\!(2)\!} & z \\
  1 &  \widehat{l}_{\!(2)\!}
  \end{matrix} &  \color{purple}{\begin{matrix} 0  \\ 0 \end{matrix}}  & \color{purple}{ \begin{matrix} 0  \\ 0 \end{matrix}} \\
  \hline
  
  \color{teal}{\begin{matrix} 0 & 0 \end{matrix}} &  \color{teal}{\begin{matrix} 0 & 0 \end{matrix}} &  l_{5} & \color{BurntOrange}{0} \\
  
  \hline
  
   \color{teal}{\begin{matrix} 0 & 0 \end{matrix}}&   \color{teal}{\begin{matrix} 0 & 0 \end{matrix}} &    \color{BurntOrange}{0}& l_6
\end{array}\right) \de z.
\end{equation}
Besides the $2\times2$ and $1\times1$ blocks on the diagonal,
there are four additional types of blocks: $2\times2$ off-diagonal blocks (shown in blue),
$1 \times 2$ off-diagonal blocks (shown in teal), $2 \times 1$ off-diagonal blocks (shown in the purple), and $1\times 1$ off-diagonal blocks (shown in orange). For notational convenience we first assume that there is a single coordinate $z$, and write
\begin{equation} \label{eq:singlez}
 \varphi =  \left(\bigoplus_{j=1}^\ell\begin{pmatrix}\widehat{l}_{(j)} & z  \\ 1& \widehat{l}_{(j)} \end{pmatrix} \oplus \begin{pmatrix} l_{n-2\ell+1} && \\ & \ddots & \\ & & l_n \end{pmatrix} \right) \de z.
\end{equation}

In the \textbf{$\mathbf{2 \times 2}$ diagonal} blocks,
\begin{equation}
 \underbrace{\begin{pmatrix} \dot P + \dot P_1 & \dot P_2 \\ \dot P_3 & \dot P- \dot P_1 \end{pmatrix}\de z}_{\dot{\varphi}} + \left[\underbrace{\begin{pmatrix}l_j & z  
 \\ 1
 & l_j \end{pmatrix} \de z}_{\varphi}, \underbrace{\begin{pmatrix} -\frac{\dot P_3}{2} & \dot P_1 \\ - 0 & \frac{\dot P_3}{2} \end{pmatrix}}_{\dot{\gamma}} \right] =\begin{pmatrix} \dot P & \dot P_2 + \dot P_3 z \\ 0 & \dot P \end{pmatrix} \de z,
\end{equation}
we observe that $\dot{\gamma}$ is indeed holomorphic. In the \textbf{$\mathbf{1 \times 1}$ diagonal} blocks with $i>2 \ell$, $\dot{\gamma}_{ii}=0$.
In the \textbf{$\mathbf{1 \times 2}$ off-diagonal} and \textbf{$\mathbf{2 \times 1}$ off-diagonal} blocks,
\begin{equation}
 \underbrace{\begin{pmatrix} 0 & 0 & \color{purple}{\dot P_{13}} \\ 0 & 0 & \color{purple}{\dot P_{23}} \\ \color{teal}{\dot P_{31}} & \color{teal}{\dot P_{32}} & 0 \end{pmatrix} \de z}_{\dot{\varphi}} + \left[\underbrace{\begin{pmatrix}l_j & z & 0 \\ 1 & l_j & 0 \\ 0 & 0 &l_k \end{pmatrix} \de z}_{\varphi}, \underbrace{\begin{pmatrix} 0 & 0 & \color{purple}{\frac{-\dot P_{23}z + \dot P_{13}\delta}{z-\delta^2}}\\                                                                                                                                                                                                                                                   0 & 0 & \color{purple}{\frac{-\dot P_{13} + \dot P_{23} \delta}{z-\delta^2}}\\
\color{teal}{\frac{\dot P_{32} - \dot P_{31}\delta}{z-\delta^2}} & \color{teal}{\frac{\dot P_{31} z - \dot P_{32}\delta}{z-\delta^2}} & 0                                                                                                                                                                                                                                                  
 \end{pmatrix}}_{\dot{\gamma}} \right] = \mathbf{0},
\end{equation}
where $\delta = l_j-l_k$.
Here, note that the denominator $\delta^2-z =((l_j +\sqrt{z}) - l_k)((l_j-\sqrt{z}) - l_k)$ never vanishes on $\mathbb{D}$ since the eigenvalues of $\varphi$ are $(l_j+\sqrt{z})\de z, (l_j - \sqrt{z})\de z,$ and $l_k\de z$.
Similarly,  in the \textbf{$\mathbf{2 \times 2}$ off-diagonal} blocks,
\begin{equation}
 \underbrace{\begin{pmatrix} 0 & 0 & \color{blue}{\dot P_{13}} & \color{blue}{\dot P_{14}} \\ 0 & 0 & \color{blue}{\dot P_{23}} & \color{blue}{\dot P_{24}} \\ 0 & 0 & 0 & 0 \\ 0 & 0 & 0 & 0 \end{pmatrix} \de z}_{\dot{\varphi}}
 +
 \left[\underbrace{\begin{pmatrix} l_j & z & 0 & 0 \\ 1 & l_j & 0 & 0 \\ 0 & 0 & l_k & z \\ 0 & 0 & 1 & l_k\end{pmatrix} \de z}_{\varphi}, \dot{\gamma}\right]= \mathbf{0}
\end{equation}
for
\begin{equation} \label{eq:2by2off}
\dot{\gamma}=\begin{pmatrix}
    0 & 0 & \color{blue}{\frac{2\dot P_{24} z + \dot P_{13}(2z - \delta^2) - (\dot P_{14}-\dot P_{23} z)\delta}{-4z \delta + \delta^3}} & \color{blue}{ \frac{2\dot P_{23} z^2 +\dot P_{14}(2z-\delta^2)-(\dot P_{13}-\dot P_{23})z \delta}{-4z \delta + \delta^3}}\\   
    0 & 0 & \color{blue}{\frac{2 \dot P_{14} + \dot P_{23}(2z-\delta^2) + (\dot P_{13}-\dot P_{24}) \delta}{-4z \delta + \delta^3}}& \color{blue}{\frac{2 \dot P_{13}z + \dot P_{24}(2z - \delta^2) + (\dot P_{14}-\dot P_{23}z)\delta}{-4z \delta + \delta^3}}\\
0 & 0 & 0 & 0 \\ 
0 & 0 & 0 & 0\end{pmatrix}
\end{equation}
where $\delta = l_j-l_k$.
Here, note that the denominator never vanishes, because $\delta=(l_j+\sqrt{z})-(l_k + \sqrt{z})=(l_j-\sqrt{z})-(l_k - \sqrt{z})$ never vanishes and \begin{equation}\delta^2-4z=\left((l_j+\sqrt{z})-(l_k+\sqrt{z}) \right)
\left((l_j-\sqrt{z})-(l_k-\sqrt{z}) \right)\end{equation} never vanishes.
Lastly, the simplest case are  the \textbf{$\mathbf{1 \times 1}$ off-diagonal} blocks, where 
\begin{equation}
\underbrace{
 \begin{pmatrix} \dot P_{11} & \color{BurntOrange}{\dot P_{12}} \\ \color{BurntOrange}{\dot P_{21}} & \dot P_{22} \end{pmatrix}\de z}_{\dot{\varphi}} + \left[\underbrace{\begin{pmatrix} l_j & 0 \\ 0 & l_k \end{pmatrix} \de z}_{\varphi}, \underbrace{\begin{pmatrix} 0 & \color{BurntOrange}{ -\frac{\dot P_{12}}{l_j-l_k}} \\ \color{BurntOrange}{ \frac{\dot P_{21}}{l_j-l_k}} & 0 \end{pmatrix}}_{\dot{\gamma}} \right] =  \begin{pmatrix} \dot P_{11} &0 \\ 0 & \dot P_{22} \end{pmatrix}\de z
\end{equation}
Consequently, $\dot{\gamma}$ defined in \eqref{eq:defofgamma} is indeed holomorphic under the assumption of \eqref{eq:singlez} that $z_1=z_2= \cdots = z_\ell$.
If there is not a single coordinate, we only need to verify that $\dot{\gamma}$ defined in \eqref{eq:defofgamma} is still holomorphic in the $2 \times 2$ off-diagonal case. We note that given two holomorphic coordinates $z_j$ and $z_k$ centered at $0$, $z_k=f(z_j)$ where $f$ is an invertible holomorphic function. Consequently, we can write everything in the holomorphic coordinate $z_j$. 
The expression in \eqref{eq:defofgamma} still defines $\dot{\gamma}$ satisfying
\begin{equation}
 \underbrace{\begin{pmatrix} 0 & 0 & \color{blue}{\dot P_{13}} & \color{blue}{\dot P_{14}} \\ 0 & 0 & \color{blue}{\dot P_{23}} & \color{blue}{\dot P_{24}} \\ 0 & 0 & 0 & 0 \\ 0 & 0 & 0 & 0 \end{pmatrix} \de z_j}_{\dot{\varphi}}
 +
 \left[\underbrace{\begin{pmatrix} l_j & z_j & 0 & 0 \\ 1 & l_j & 0 & 0 \\ 0 & 0 & l_k & f(z_j) f'(z_j) \\ 0 & 0 & f'(z_j) & l_k\end{pmatrix} \de z_j}_{\varphi}, \dot{\gamma}\right]= \mathbf{0},
\end{equation}
however the expression for $\dot{\gamma}$ is considerably uglier that \eqref{eq:2by2off}.
The denominator of terms of $\dot{\gamma}$ is
\begin{equation}
-2 f(z_j) f'(z_j)^2 \left((l_j-l_k)^2+z_j\right)+f(z_j)^2 f'(z_j)^4+\left(z_j-(l_j-l_k)^2\right)^2,
\end{equation}
which does not vanish on $\bD$ since it is the product of the four differences of the eigenvalues of $\varphi$:
\begin{equation}
 \left((l_j \pm \sqrt{z_j}) - (l_k \pm \sqrt{f(z_j)}f'(z_j)) \right) \de z_j.
\end{equation}
Thus, there is representative with $(\dot{\eta}, \dot{\varphi})$ as claimed in \eqref{eq:defshapesln}.

\medskip

To see that we can take representative $\dot{\nu}_\infty$ as in \eqref{eq:defshapesln}, note
that from the infinitesimal version of $[\varphi, \varphi^{\dagger_{h_\infty}}]=0$ given in \eqref{eq:infinitesimaldecoupled}, $\dot{\nu}_\infty$ is block diagonal.  By similar arguments to the $n=2$ case, 
 we can use some of the remaining gauge freedom in each $2 \times 2$ block to arrange that
 \begin{equation}
  \dot{\nu}_\infty^{\Delta} = 
\bigoplus_{j=1}^\ell \left(
  -\frac{\widehat{\dot{P}}_{(j)}}{4z_j} \begin{pmatrix}1& 0 \\ 0 & -1 \end{pmatrix}  + \widehat{\dot{c}}_{(j)} \begin{pmatrix} 1 & 0 \\ 0 & 1 \end{pmatrix}\right) 
  \oplus \begin{pmatrix} \dot{c}_{2\ell + 1} & & \\ & \ddots & \\ & & \dot{c}_{n}  \end{pmatrix}.
 \end{equation}
for $\widehat{\dot{c}}_{(j)}$ and $\dot{c}_{j}$ holomorphic. Finally, take the infinitesimal complex gauge transformation
\begin{equation} \label{eq:diag}
 \dot{\gamma} =  \bigoplus_{j=1}^\ell  \begin{pmatrix} -\widehat{\dot{c}}_{(j)} & 0 \\ 0 & -\widehat{\dot{c}}_{(j)} \end{pmatrix} \oplus  \begin{pmatrix} -\dot{c}_{2\ell + 1} & & \\ & \ddots & \\ & & -\dot{c}_{n}  \end{pmatrix} ,
\end{equation}
and observe that it is holomorphic and satisfies $[\varphi, \dot{\gamma}]=0$, hence this infinitesimal gauge action doesn't change the expression for $\dot{\eta}$ or $\dot{\varphi}$. However, it changes the infinitesimal deformation of the hermitian metric 
\begin{equation}
\dot{\nu}_\infty \mapsto \dot{\nu}_\infty + \dot{\gamma}=\bigoplus_{j=1}^\ell \left(
  -\frac{\widehat{\dot{P}}_{(j)}}{4z_j} \begin{pmatrix}1& 0 \\ 0 & -1 \end{pmatrix}  \right) 
  \oplus \mathbf{0}_{n-2\ell};
\end{equation}
consequently, there is a representative with $(\dot{\eta}, \dot{\varphi}, \dot{\nu}_\infty)$ as claimed.

\medskip

\noindent\textsc{Claim 2:} \emph{The variation $\dot{\nu}_t^{\app}$ is
\begin{equation}
  \dot{\nu}_t^{\app} = 
  -\frac{1}{2} F_t^{(j)} \begin{pmatrix}1& 0 \\ 0 & -1 \end{pmatrix} 
  \oplus \mathbf{0}_{n-2\ell}.
\end{equation}
}\\ $\triangleright$
The variation $\dot{\nu}_t^\app$ solves \eqref{eq:7app} and agrees with $\dot \nu_\infty$ wherever $h_t^\app$ is equal to $h_\infty$. From Corollary \ref{cor:dirichlet}, this  determines $\dot{\nu}_t^{\app}$ uniquely. In the equation \eqref{eq:7app}, 
the blocks shown in \eqref{eq:colored} all decouple.  Away from the $\ell$ $2\times2$ blocks on the  diagonal, $(\dot{\nu}_t^\app)_{ij}=0$ is a solution, hence (by uniqueness) it is \emph{the} solution.
It follows that $\dot{\nu}_t^{\app}$ is block diagonal with $\mathbf{0}_{n-2\ell}$ in
the bottom right corner. 
It remains to determine $\dot{\nu}_t^{\app}$ in each of the $\ell$ $2\times 2$ blocks.
In each $2\times2$ block, we use the same $n=2$ argument in Theorem \ref{thm:local}, and we conclude that
\begin{equation}
  \dot{\nu}_t^{\app, \Delta} = 
\bigoplus_{j=1}^\ell
 - \frac{1}{2} F_t^{(j)} \begin{pmatrix}1& 0 \\ 0 & -1 \end{pmatrix}  \oplus  \mathbf{0}_{n-2\ell},
\end{equation}
where $F_t^{(j)}$ is a complex-valued function. 
 $\triangleleft$

\bigskip

We turn now to proving the exponential decay in \eqref{eq:propnearSLn}.
In order to prove the exponential decay of 
\begin{eqnarray*}
\norm{(0, t\dot{\varphi}, \dot{\nu}_t^\app)}^2_{g_{\app}(\bD)}-\norm{(0, t\dot{\varphi}, \dot{\nu}_\infty)}^2_{g_{\semif}(\bD)}&=& 2t^2\int_C  
 \IP{\dot{\varphi} +[\dot{\nu}_t^\app, \varphi], \dot{\varphi}}_{h_t^\app}-
  \IP{\dot{\varphi} +[\dot{\nu}_\infty, \varphi], \dot{\varphi}}_{h_\infty},
 \end{eqnarray*}
 we observe that this integral is a sum of the contributions from each of the $\ell$ $2\times 2$ blocks.  In Theorem \ref{thm:local}, we proved
 that the contribution from the $k^{th}$ $2\times 2$ block was exponentially-decaying like $O(\e^{-\gamma_k t})$. Taking $\gamma=\min(\gamma_1, \cdots, \gamma_\ell)$, it follows that
\begin{equation} 
\norm{(0, t\dot{\varphi}, \dot{\nu}_t^\app)}^2_{g_{\app}(\bD)}-\norm{(0, t\dot{\varphi}, \dot{\nu}_\infty)}^2_{g_{\semif}(\bD)} =
O(\e^{-\gamma t}).
\end{equation}
\end{proof}

\begin{proof}[Proof of Theorem \ref{7thm:mainsln}]
The structure of the proof is the same as the $n=2$ case. 
We  
break the difference $ g_{L^2}(\dot{\psi}_t, \dot{\psi}_t) - g_{\semif}(\dot{\psi}_t, \dot{\psi}_t)$ into two pieces:
 \begin{eqnarray}\label{eq:twopiecessln}
   \left( g_{L^2}(\dot{\psi}_t, \dot{\psi}_t) - g_{\app}(\dot{\psi}_t, \dot{\psi}_t)  \right) +   \left( g_{\app}(\dot{\psi}_t, \dot{\psi}_t) - g_{\semif}(\dot{\psi}_t, \dot{\psi}_t)  \right).
 \end{eqnarray}
 For the first piece of \eqref{eq:twopiecessln},
 the author proves in \cite{FredricksonSLn} that
 the harmonic metric $h_t$ is close to $h_t^\app$; in particular,
 \begin{equation}h_t(w_1,w_2)= h_t^\app(\e^{-\kappa_t} w_1, \e^{-\kappa_t} w_2)
 \end{equation}
for $h_0$-hermitian $\kappa_t$   satisfying
$\|\kappa_t\|_{H^2\left(i \mathfrak{su}(E)\right)} \leq C \e^{-\delta t}$.
 By a similar argument to the $n=2$ case, it follows that there is a constant $\gamma$ such that 
\begin{equation}
     g_{L^2}(\dot{\psi}_t, \dot{\psi}_t) - g_{\app}(\dot{\psi}_t, \dot{\psi}_t)  =O(\e^{-\gamma t}).
\end{equation}
For the second piece of \eqref{eq:twopiecessln}, 
first note that $g_{\semif}$ is the $L^2$-metric on $\cM'_\infty$ by Theorem \ref{thm:semiflatisL2}.
 The family of approximate metrics $h_t^\app$ constructed in \cite{FredricksonSLn} differ from $h_\infty$ only on disks around the branch points $Z$; consequently, the only contribution to the difference $ g_{\app}(\dot{\psi}_t, \dot{\psi}_t) - g_{\semif}(\dot{\psi}_t, \dot{\psi}_t)$ is from these disks. By Proposition \ref{thm:localSLn}, 
\begin{equation}
  g_{\app}(\dot{\psi}_t, \dot{\psi}_t) - g_{\semif}(\dot{\psi}_t, \dot{\psi}_t) = O(\e^{-\gamma t}).
\end{equation}
\end{proof}

\bibliography{ends}{}
\bibliographystyle{fredrickson}
\end{document}